\documentclass[11pt]{article}
\usepackage{graphicx}
\usepackage{epstopdf,epsfig}

\usepackage{palatino}
\usepackage[mathcal]{euler}

\usepackage{amsmath,amsthm,amsfonts,amssymb,latexsym,amscd,enumerate}

\usepackage{lscape}

\usepackage{xcolor}

\usepackage[capitalise]{cleveref}

\usepackage{xy}
\xyoption{all}

\theoremstyle{plain}
    \newtheorem{theorem}{Theorem}[section]
    \newtheorem{lemma}[theorem]{Lemma}
    \newtheorem{corollary}[theorem]{Corollary}
    \newtheorem{proposition}[theorem]{Proposition}

 \theoremstyle{definition}

    \newtheorem{example}[theorem]{Example}

    \newtheorem{remark}[theorem]{Remark}

\theoremstyle{remark}

\numberwithin{equation}{section}

    \newcommand{\R}{\mathbb{R}}
    \newcommand{\C}{\mathbb{C}} 
    \newcommand{\N}{\mathbb{N}}
    \newcommand{\Z}{\mathbb{Z}}

    \DeclareMathOperator{\supp}{supp}
\DeclareMathOperator{\Res}{Res}
 
 \DeclareMathOperator{\Tr}{Tr}
  \DeclareMathOperator{\Cl}{Cl}

\newcommand{\kg}{\mathfrak{g}} 
\newcommand{\kk}{\mathfrak{k}}  
\newcommand{\ka}{\mathfrak{a}}   
  
\newcommand{\km}{\mathfrak{m}}   
     
\newcommand{\kt}{\mathfrak{t}}
\newcommand{\kh}{\mathfrak{h}} 
\newcommand{\kn}{\mathfrak{n}} 
\newcommand{\kp}{\mathfrak{p}}

\newcommand{\cH}{\mathcal{H}}
\newcommand{\cB}{\mathcal{B}}

\newcommand{\cL}{\mathcal{L}}
\newcommand{\cC}{\mathcal{C}}

\newcommand{\cN}{\mathcal{N}}

\DeclareMathOperator{\SL}{SL}
\DeclareMathOperator{\SO}{SO}

\DeclareMathOperator{\Spin}{Spin}

\DeclareMathOperator{\SU}{SU}

\DeclareMathOperator{\U}{U}

\DeclareMathOperator{\vol}{vol}

\DeclareMathOperator{\Ad}{Ad}

\DeclareMathOperator{\reg}{reg}
\DeclareMathOperator{\rank}{rank}

\DeclareMathOperator{\ind}{ind}
\DeclareMathOperator{\Ind}{Ind}
\DeclareMathOperator{\temp}{temp}

\DeclareMathOperator{\spec}{spec}
\DeclareMathOperator{\HS}{HS}

\DeclareMathOperator{\even}{even}

\DeclareMathOperator{\ch}{ch}

\DeclareMathOperator{\End}{End}
\DeclareMathOperator{\Hom}{Hom}

\DeclareMathOperator{\DInd}{D-Ind}

\newcommand{\Spinc}{\Spin^c}

\newcommand{\beq}[1]{\begin{equation} \label{#1}}
\newcommand{\eeq}{\end{equation}}

\begin{document}

\title{An index theorem for higher orbital integrals}  
\date{\today}
\author{Peter Hochs,\footnote{Radboud University, University of Adelaide, \texttt{p.hochs@math.ru.nl}} {} 
Yanli Song\footnote{Washington University in St.\ Louis, \texttt{yanlisong@wustl.edu}}  {} and
 Xiang Tang\footnote{Washington University in St.\ Louis, \texttt{xtang@wustl.edu}} 
 }  
\maketitle

\begin{abstract}
Recently, two of the authors of this paper constructed cyclic cocycles on Harish-Chandra's Schwartz algebra of linear reductive Lie groups that detect all information in the $K$-theory of the corresponding group $C^*$-algebra. The main result in this paper is an index formula for the pairings of these cocycles with equivariant indices of elliptic operators for proper, cocompact actions. This index formula completely determines such equivariant indices via topological expressions.
\end{abstract}

\setcounter{tocdepth}{2}

\tableofcontents

\section{Introduction}

Consider a locally compact group $G$, and a proper, isometric action by $G$ on a Riemannian manifold $X$, such that $X/G$ is compact. Let $D$ be a $G$-equivariant, elliptic differential operator on $X$. Then $D$ has an equivariant index in the $K$-theory of the reduced group $C^*$-algebra $C^*_r(G)$,
\beq{eq G index}
\ind_G(D) \in K_*(C^*_r(G)),
\eeq
defined by the analytic assembly map \cite{Connes94}. This index generalises the equivariant index in the compact case. It has a range of applications, e.g.\ to geometry and topology via the Baum--Connes and Novikov conjectures; to questions about positive scalar curvature; to representation theory \cite{HW2}; and to geometric quantisation \cite{Landsman05}.

It is a natural question to extract numerical invariants from the index \eqref{eq G index} and compute them. This is relevant, for example, to detect the vanishing of the index, and because such numbers may have a meaning in geometry, topology or representation theory. A natural way to extract numbers from $K$-theory classes is by pairing them with \emph{traces} or more general \emph{cyclic cocycles} on the algebra in question, or an algebra with the same $K$-theory. For suitable convolution algebras on groups, such traces can be defined by \emph{orbital integrals}: integrals over conjugacy classes. For the trivial conjugacy class, one then obtains the classical von Neumann trace. It was shown in various places, see e.g. \cite{Gong, HW2, HW19,  Samurkas}, that traces defined by orbital integrals over nontrivial conjugacy classes can yield more, and relevant, information than the von Neumann trace.

In this paper, we focus on connected real linear reductive Lie groups. 
Connes and Moscovici \cite{Connes82} showed that the von Neumann trace is then only nonzero on $K$-theory classes defined by discrete series representations. So this trace is not enough to detect all information about $K$-theory classes, especially for groups with no discrete series representations. In \cite{HW19}, it is shown that one does detect all information in $K_*(C^*_r(G))$ for groups with discrete series representations if one uses traces defined by more general orbital integrals. In another direction, higher cyclic cocycles coming from group cohomology were used to extract and compute relevant numbers from the index \eqref{eq G index} in  \cite{Pflaum15}.

Recently, two of the authors of this paper developed a family of cyclic cocycles on Harish--Chandra's Schwartz algebra \cite{ST19},  for linear reductive $G$. These cocycles are higher versions, in a sense, of orbital integrals. They  detect all information about classes in $K_*(C^*_r(G))$, in the sense that if the pairings of a class in $K_*(C^*_r(G))$ with all these cyclic cocycles are zero, then it is the zero class.
%
%
 This means that an index formula for the pairing of \eqref{eq G index} with these cocycles is a complete topological description of this index. Our goal  in this paper is to prove such an index formula, Theorem \ref{thm index}.


At the end of \cite{Connes82}, Connes and Moscovici write that they consider obtaining an `intrinsic' index formula for \eqref{eq G index} for Lie groups $G$, to be a problem that deserves further study.
We believe that Theorem \ref{thm index} is  such an intrinsic index formula, detecting all relevant information, for linear reductive Lie groups.
This builds on a long development, and results by many authors, see Table \ref{table results}. This type of index theorem for pairings with higher cocycles, as well as the techniques involved,  was heavily inspired by the work of Connes and Moscovici for discrete groups, see Theorem 5.4 in \cite{CM90}.


\begin{table}
\begin{tabular}{|l|l|l|}
\hline
& trivial conjugacy class $\{e\}$ & more general conjugacy classes \\
\hline
trace & 
\begin{tabular}{l}
Connes--Moscovici 1982 \cite{Connes82} \\(homogeneous spaces)\\
Wang 2014 \cite{Wang14}\\ (general case)
\end{tabular}
&  
\begin{tabular}{l}
Hochs--Wang 2017 \cite{HW2}\\ (semisimple Lie groups)
\end{tabular}
\\
\hline
higher cocycles & 
\begin{tabular}{l}
 Pflaum--Posthuma--Tang 2015 \cite{Pflaum15} \\(cocycles from group cohomology)
 \end{tabular}&
 \begin{tabular}{l}
  Theorem \ref{thm index} \\
  (linear reductive groups)
  \end{tabular} \\
\hline
\end{tabular}
\caption{Some equivariant index theorems for Lie groups $G$,  for pairings of \eqref{eq G index} with traces and higher cocycles associated to orbital integrals}
\label{table results}
\end{table}

Theorem \ref{thm index} is stated for twisted $\Spinc$-Dirac operators. At the end of this paper, we describe how to generalise this to arbitrary elliptic operators; obtain a higher version of Connes and Moscovici's $L^2$-index theorem for homogeneous spaces; obtain  independent proofs of injectivity of Dirac induction and the fact that the cocycles from \cite{ST19} detect all information from $K_*(C^*_r(G))$; realise the generators of $K_*(C^*_r(G))$ constructed in \cite{ST19} as indices; and work out the case of complex semisimple groups.

Other applications of Theorem \ref{thm index} are of the same type as those in \cite{Pi-Po}. 
\begin{itemize}
\item When $D$ is the signature operator on $X$, the index element $\ind_G(D) $ is invariant under $G$-equivariant homotopy \cite[Theorem A]{Fukumoto}. Our theorem \ref{thm index} shows that the corresponding generalized signature number is invariant under $G$-equivariant homotopy. 
\item When $D$ is the Spin Dirac operator on a spin manifold $X$, the non-vanishing of the generalized $\hat{A}$-class on the right side would be an obstruction for $X$ to carry a $G$-invariant metric with positive scalar curvature.
\end{itemize}
We plan to investigate potential applications of our theorem in representation theory in the future.  

Let us briefly go into the idea of the proof of Theorem \ref{thm index}. The cocycles we use are associated to parabolic subgroups $P = MAN<G$, where $M$ is reductive, $A$ is isomorphic to $\R^l$ and $N$ is nilpotent.
The proof of Theorem \ref{thm index} is based on an equality between  the pairing of \eqref{eq G index} with these cocycles  and an analogous pairing of an index of an $M$-equivariant operator on $X/AN$ with an analogous cocycle for the parabolic subgroup $M<M$. Because $M$ has a compact Cartan  for the parabolic subgroups we use, the latter pairing can be computed via results from \cite{HW2}. It is an interesting question to what extent the transition between these $G$ and $M$-equivariant indices is related to parabolic induction. Some indication for this is given in Subsection \ref{sec diagram}.

\subsection*{Acknowledgements}

Hochs is partially supported by Discovery Project DP200100729 from the Australian Research Council; 
Song is partially supported by NSF grant DMS-1800667; Tang is partially supported by NSF grants DMS-1363250 and DMS-1800666.

\section{An index theorem}

Let $G$\label{pag G} be a connected, linear, real reductive Lie group. Let $X$\label{pag X} be a Riemannian manifold on which $G$ acts properly, cocompactly and isometrically. 

Fix a $G$-equivariant $\Spinc$-structure on $X$, assuming it exists. Let $S \to X$\label{pag S} be the corresponding spinor bundle. Let $W \to X$\label{pag W} be a $G$-equivariant, Hermitian vector bundle, and write $E := S\otimes W$.\label{pag E} Let $D$\label{pag D} be a twisted $\Spinc$-Dirac operator acting on sections $E$, associated to a $G$-invariant Clifford connection on $S$ and a $G$-invariant Hermitian connection on $W$. Because $D$ is an elliptic, $G$-equivariant operator, and $X/G$ is compact, we have the index
\beq{eq indG}
\ind_G(D) \in K_*(C^*_r(G)),
\eeq
defined via the analytic assembly map \cite{Connes94}. Here $\ind_G(D)$ lies in even $K$-theory if $X$ is even-dimensional, so that $S$ admits a natural grading and $D$ is odd, and in odd $K$-theory if $X$ is odd-dimensional.

Our goal is to obtain a topological expression for the pairing of \eqref{eq indG} with  natural cyclic cocycles constructed in \cite{ST19}.

\subsection{Cocycles}

We use a lower case gothic letter to denote the Lie algebra of the Lie group denoted by the corresponding upper case Roman letter. 

Let $K<G$ be maximal compact.\label{pag K} 
Let  $P<G$\label{pag P} be a parabolic subgroup. It has a Langlands decomposition $P=MAN$, where $M$\label{pag M} is reductive, $A$\label{pag A} is isomorphic to a vector space, and $N$\label{pag N} is nilpotent. 
We assume from now on that $P$ is a cuspidal parabolic subgroup. (See Section \ref{sec cuspidal} for the non-cuspidal case.)
This means that $M$ has a compact Cartan subgroup, which implies that it has discrete series representations. In several places, we will use the facts that $N$ normalises $MA$  and that $M$ and $A$ commute.

 
 Let $dg$, $dk$, $dm$, $da$ and $dn$ be Haar measures on the groups $G$, $K$, $M$, $A$ and $N$, respectively, such that
\beq{eq decomp dg}
dg = dk\, dm\, da\, dn
\eeq
according to the decomposition $G = KMAN$. All these groups are unimodular, so these measures are left- and right-invariant.
 
 Define the map  $H\colon G \to \ka$\label{pag H} by the property that for all $g \in G$, $\exp(H(g))$ is the component of $g$ in $A$ according to the decomposition $G = KMAN$.
 Set $l:= \dim (A)$.\label{pag l}
 If a basis of $\ka$ is chosen, then   for $a_1, \ldots, a_l \in \ka$, the vectors $a_1, \ldots, a_l$ together define 
  an $l \times l$ matrix. So we can take its determinant $\det(a_1, \ldots,a_l)$. This determinant depends on the choice of basis, but if $x_1, \ldots, x_n \in \ka^*$ are the coordinates on $\ka$ defined by this basis, then the volume element $\det(a_1, \ldots,a_l) dx_1 \wedge \cdots \wedge dx_l \in \bigwedge^l\ka^*$, which is the relevant object in what follows, does not. For definiteness' sake, we fix a basis such that 
 \beq{eq vol A}
  dx_1 \wedge \cdots \wedge dx_l = da,
 \eeq
   the Haar measure on $A \cong \ka$, and use this to define $\det(a_1, \ldots,a_l)$ for any $a_1, \ldots, a_l \in \ka$.
 
 In \cite{ST19}, a cyclic cocycle on \emph{Harish-Chandra's Schwartz algebra} $\cC(G)$\label{pag CG} is constructed. This algebra is a generalisation of the algebra of Schwartz functions on $\R^n$, and is defined as follows.
 Let $\pi_0$ be the unitary representation of $G$ induced from the trivial representation of a minimal parabolic subgroup. Let $\xi$ be a unit vector in the representation space of $\pi_0$, fixed by $\pi_0(K)$. Let $\Xi$ be the function on $G$ defined by 
\[
\Xi(g) = (\xi, \pi_0(g)\xi)
\]
for all $g \in G$. 
The inner product on $\kg$ defines a $G$-invariant Riemannian metric on $G/K$. For $g \in G$, let $d(g)$ be the Riemannian distance from $eK$ to $gK$ in $G/K$. 
Then $\cC(G)$ is the space of $f \in C^{\infty}(G)$ such that for all $m \geq 0$ and $X,Y \in U(\kg)$, 
\beq{eq seminorms CG}
 \sup_{g \in G} (1+d(g))^m \Xi(g)^{-1} |L(X)R(Y)f(g)| < \infty,
\eeq
where $L$ and $R$ denote the left and right regular representations, respectively.
See Section 9 in \cite{HC66}. The space $\cC(G)$ is a Fr\'echet algebra with respect to the seminorms \eqref{eq seminorms CG}. 
 

  

 Let $x \in M$\label{pag x} be a semisimple element. Let $Z := Z_M(x)$\label{pag Z} be its centraliser in $M$. Because $x$ is semisimple, the quotient $M/Z$ has a $M$-invariant measure $d(hZ)$ (also denoted by $d(mZ)$) compatible with the Haar measure $dm$ on $M$.
In Definition 3.3 in \cite{ST19}, the cyclic $l$-cocycle $\Phi^P_x$ on $\cC(G)$ is defined by\label{pag PhiPx}
\begin{multline} \label{eq def cocycle}
\Phi^P_x(f_0, \ldots, f_l) := \int_{M/Z} \int_{KN} \int_{G^l} \\
\det \bigl(H(g_1 g_2 \cdots g_l k), H(g_2 \cdots g_l k)), \ldots, H(g_{l-1}g_l k), H(g_lk)\bigr)\\
f_0(khxh^{-1}nk^{-1} (g_1 \cdots g_l)^{-1}) f_1(g_1) \cdots f_l(g_l)
\, dg_1 \cdots dg_l\, dk\, dn\, d(hZ),
\end{multline}
for $f_0, \ldots, f_l \in \cC(G)$. Theorem 3.5 in \cite{ST19} states that this is indeed a cyclic cocycle.

If   $G$ has a compact Cartan subgroup, then we can take $P = M = G$, so that
If $l = 0$. Then $\Phi^P_x$ is simply the orbital integral trace associated to $x$:
\[
\Phi^G_x(f_0) = \int_{G/Z} 
f_0(hxh^{-1})
 d(hZ).
\]

\subsection{An index pairing}

The Schwartz algebra $\cC(G)$ is a dense subalgebra of $C^*_r(G)$, closed under holomorphic functional calculus, see Theorem 2.3 in \cite{HW2}.  So $\ind_G(D) \in K_*(\cC(G))$, and we obtain the number
\[
\langle \Phi^P_x, \ind_G(D) \rangle 
\]
by the pairing between cyclic cohomology and $K$-theory. Our main result, Theorem \ref{thm index} below, is a topological expression for this number.

 The group $AN$ has no nontrivial compact subgroups and acts properly on $X$, so the action is free, and $X/AN$ is a smooth manifold.  Because $M$ commutes with $A$ and is normalised by $N$, the action by $x$ on $X$ preserves $NA$-orbits, and hence descends to an action on $X/NA$. Also, if $F \to X$ is a $G$-equivariant vector bundle, then $F/AN \to X/AN$ is an $M$-equivariant vector bundle.

Consider the fixed point set $(X/AN)^x \subset X/AN$. The connected components of $(X/AN)^x$ are submanifolds of $X/AN$ of possibly different dimensions.
We apply all constructions below to the connected components of $(X/AN)^x$ and add the results together. 

The set $(X/AN)^x$ is preserved by the action of the centraliser $Z$ of $x$ in $M$.
Let $dz$ be the Haar measure on $Z$ such that for all $\varphi \in C_c(M)$,
\[
\int_M \varphi(m)\, dm = \int_{M/Z} \int_Z \varphi(zm)\, dz\, d(mZ).
\] 
Let $\chi_x$\label{pag chix} be a smooth, compactly supported function on $(X/AN)^x$ such that for all $p \in (X/AN)^x$,
\[
\int_Z \chi_x(zp)\, dz = 1.
\]

The $\Spinc$-structure on $X$ induces an $M$-equivariant $\Spinc$-structure on $X/AN$ with spinor bundle $S_{AN}$\label{pag SAN} such that
\beq{eq SAN}
S/AN = S_{AN} \otimes F_{AN} \otimes S_{\ka} \to X/AN,
\eeq
for an $M$-equivariant, graded vector bundle $F_{AN} \to X/AN$,\label{pag FAN} and a vector space $S_{\ka}$
of dimension $2^{\lfloor \frac{\dim(\ka) - 1}{2}\rfloor}$. See Subsection \ref{sec cocycles MA} for an explicit construction, in particular \eqref{eq SAN FAN}. Let $L_{\det} \to X/AN$\label{pag Ldet} be the determinant line bundle of the $\Spinc$-structure on $X/AN$ with spinor bundle $S_{AN}$.

%

 Let $\hat A((X/AN)^x)$ be the $\hat A$-class of $(X/AN)^x$. Let $\cN \to (X/AN)^x$\label{pag cN} be the normal bundle to $(X/AN)^x$ in $X/AN$.  Let $R^{\cN}$ be the curvature of the Levi--Civita connection on $X/AN$ restricted to $\cN$.  

Consider the vector bundle\label{pag WAN}
\[
W_{AN} := F_{AN} \otimes W/AN \to X/AN.
\]
If the closure $T_x$ of the set of powers of $x$ is a compact subgroup of $M$, then 
the restriction of $W_{AN}$ to the compact set $\supp(\chi_x) \subset (X/AN)^x$ defines a class
\[
[W_{AN}|_{\supp(\chi_x)}] \in K^0_{T_x}(\supp(\chi_x)) \cong K^0(\supp(\chi_x)) \otimes_{\Z} R(T_x),
\]
where $R(T_x)$ is the representation ring of $T_x$.
 The last equality holds because $T_x$ acts trivially on $(X/AN)^x$. Evaluating the factor in $R(T_x)$ at $x$, we obtain
 \[
 [W_{AN}|_{\supp(\chi_x)}](x) \in K^0(\supp(\chi_x)) \otimes_{\Z} \C.
 \]
 Consider the Chern character
 \[
 \ch\colon K^0(\supp(\chi_x)) \otimes \C \to H^{\even}(\supp(\chi_x); \C).
 \]
 
 Suppose $P$ is a cuspidal parabolic.  Let $T<K \cap M$\label{pag T} be a maximal torus. 
\begin{theorem}[Index theorem for higher orbital integrals] \label{thm index}
For all $x \in T$,
\beq{eq index}
\langle \Phi^P_x, \ind_G(D) \rangle 
= 
\int_{(X/AN)^x} \chi_x \frac{\hat A((X/AN)^x) \ch([W_{AN}|_{\supp(\chi_x)}](x)) e^{c_1(L_{\det}|_{(X/AN)^x})}}{\det(1-x e^{-R^{\cN}/2\pi i})^{1/2}}.
\eeq
If $P$ is not a maximal cuspidal parabolic subgroup or  $x$ does not lie in a compact subgroup of $M$, then the left hand side equals zero.
\end{theorem}


\begin{remark}\label{rem SAN unique}
The $M$-equivariant $\Spinc$-structure on $X/AN$ with spinor bundle $S_N$ satisfying \eqref{eq SAN} is not unique: one can tensor $S_{AN}$ by an $M$-equivariant line bundle and $F_{AN}$ by the dual line bundle. But the right hand side of \eqref{eq index} does not change under this modification, so \eqref{eq SAN} is enough to state Theorem \ref{thm index}. For the sake of definiteness, we will use the bundles $S_{AN}$ and $F_{AN}$ defined in \eqref{eq SAN FAN}.
\end{remark}

\subsection{A reformulation on $X/N$}

The fixed point formula in Theorem \ref{thm index} can be reformulated in terms of the manifold $X/N$ rather than $X/AN$. The manifold $X/N$ has a $\Spinc$-structure with spinor bundle $S_N$ 
such that
\beq{eq SN}
S/N = S_N \otimes F_N \to X/N
\eeq
for an $M$-equivariant vector bundle $F_N \to X/N$. 
See Subsection \ref{sec index XN} for an explicit construction, in particular \eqref{eq SN FN}. Analogously to Remark \ref{rem SAN unique}, any ambiguity in the choice of $S_N$ does not affect the index formula.
Let $L_{\det}^{N}\to X/N$\label{pag LdetN} be the determinant line bundle of this $\Spinc$-structure.
%
Consider the vector bundle\label{pag WN}
\[
W_{N} := F_N \otimes W/N \to X/N.
\]
For $x \in M$, 
let $\cN_{N} \to (X/N)^x$ be the normal bundle of $(X/N)^x$ in $X/N$.

We view $da$ as the volume form \eqref{eq vol A}. Consider the projection map  $q_A\colon X/N = X/AN \times A \to A$.\label{pag qA} Let $\chi_A \in C^{\infty}_c(A)$\label{pag chiA} be such that 
\[
\int_{A} \chi_A \, da = 1.
\]
Let $q\colon X/N \to X/AN$\label{pag q} be the quotient map, and set $\chi_{x, A} := q^*\chi_x \, q_A^*\chi_A$.
\begin{lemma}\label{lem index XN}
If $P$ is a maximal cuspidal parabolic subgroup and $x \in T$, then the right hand side of \eqref{eq index} equals
\beq{eq index XN}
\int_{(X/N)^x}\chi_{x, A} \frac{\hat A((X/N)^x) \ch([W_N|_{\supp(\chi_{x, A})}](x)) e^{c_1(L_{\det}^{N}|_{(X/N)^x})}}{\det(1-x e^{-R^{\cN_{N}}/2\pi i})^{1/2}}\wedge q_A^*da.
\eeq
\end{lemma}

\begin{remark}
If $x \in T^{\reg}$, then $(X/N)^x \cong X^x$, and \eqref{eq index XN} can be rewritten as an integral over $X^x$. But it seems to be a nontrivial exercise to rewrite the integrand in terms of the manifold $X$ instead of $X/N$.
%
\end{remark}

\subsection{Idea of the proof and relation with parabolic induction}\label{sec diagram}

An intuitive outline of the proof of Theorem \ref{thm index} is given in Diagram \eqref{eq diagram proof}. 
The  ingredients of this diagram are described below.
\beq{eq diagram proof}
\xymatrix{
K_*^G(X) \ar[r]^-{\ind_G} \ar[d]_-{\Res^G_{MA}}& K_*(C^*_{r}(G))\ar[d]_-{\relbar \otimes_{C^*_r(G)} \Ind^G_P}
\ar@/_-1.5pc/[dr]^-{\langle \Phi_x^P, \relbar \rangle }
& \\
K_*^{MA}(X/N) \ar[r]^-{\ind_{MA}} & K_*(C^*_{r}(MA)) \ar[r]^-{\langle \Phi_x^{MA}, \relbar \rangle } & \C\\
K_*^M(X/AN) \ar[r]^-{\ind_M}  \ar[u]^-{\relbar \otimes [D_A]} & \ar[u]^-{\relbar \otimes \ind_A(D_A)} K_*(C^*_{r}(M))
\ar@/_0pc/[ur]_-{\langle \Phi_x^{M}, \relbar \rangle } & 
}
\eeq

\begin{enumerate}
\item
The map
\beq{eq Res G MA}
\relbar \otimes_{C^*_r(G)} \Ind^G_P\colon 
K_*(C^*_r(G)) \to K_*(C^*_r(MA))
\eeq
is defined by the tensor product over $C^*_r(G)$ from the right by a Hilbert $C^*_r(G)$-$C^*_r(MA)$ bimodule $\Ind^G_P$ defined in  \cite{Clare13,CCH16}. This module is closely related to parabolic induction, see Corollary 1 in \cite{Clare13} and Definition 4.4 in \cite{CCH16}.
%
 It is an interesting but seemingly nontrivial question  if the top right part of Diagram \eqref{eq diagram proof} commutes.
\item
There is a natural way to define a map $\Res^G_{MA}\colon K_*^G(X) \to K_*^{MA}(X/N)$, by first restricting group actions from $G$ to $P$ and then applying a construction from Section 3.2 of Valette's part of  \cite{MV03} or Appendix A in \cite{Landsman08}. One could expect the top left part of \eqref{eq diagram proof} to commute, based on results in \cite{Landsman08,MSW19,MV03}.
\item Let $D_A$ be the $\Spin$-Dirac operator on $A \cong \R^l$. The bottom left part of Diagram \eqref{eq diagram proof} is defined in terms of the exterior Kasparov product by the $K$-homology and index classes of $D_A$. This commutes by basic properties of the index; see also Lemma \ref{lem decomp DMA}.
\item The group $M$ has a compact Cartan subgroup. Hence its maximal cuspidal parabolic subgroup is $M$ itself. Similarly, $MA$ is its own maximal cuspidal parabolic subgroup. Let $\Phi^M_x$ and $\Phi^{MA}_x$ be the corresponding versions of the cocycle \eqref{eq def cocycle}. The maps on the right hand side of \eqref{eq diagram proof} are defined by pairing with the respective cocycles. The bottom right part of Diagram \eqref{eq diagram proof} commutes by Lemmas \ref{lem pairing product psi} and \ref{lem phi A DA}. 
\end{enumerate}
The open questions in points 1 and 2 are topics of ongoing research. Positive answers would make a proof of Theorem \ref{thm index} possible based on Diagram \eqref{eq diagram proof}. Indeed, using commutativity of \eqref{eq diagram proof}, one shows that the left hand side of \eqref{eq index} equals the pairing of $\Phi^M_x$ with the $M$-equivariant index of a certain Dirac operator on $X/AN$. The latter pairing can be computed via results from \cite{HW2}, because $M$ has a compact Cartan subgroup.

An indication that the questions in points 1 and 2 are nontrivial is that it seems necessary for the diagram to commute (indeed, even for some of its components to be well-defined) that $P$ is a parabolic subgroup, not just any cocompact subgroup.
We give a different and more direct  proof of Theorem \ref{thm index} here, but in spirit it is analogous to Diagram \ref{eq diagram proof}. We have included this diagram to illustrate the idea of the proof, and its possible relation to the parabolic induction map from \cite{Clare13, CCH16}.

\section{The index pairing on $X$}

The most important step in the proof of Theorem \ref{thm index} is the fact that the left hand side of \eqref{eq index} equals an index pairing on $X/N$. This is Proposition \ref{prop ind G MA}, which is analogous to commutativity of the top two diagrams in \eqref{eq diagram proof}. Proving that proposition is our main goal in this section and the next.

In this section, we prepare for the proof of Proposition \ref{prop ind G MA} by obtaining a relation between cocycles defined on $G$ and on $MA$, and giving an explicit representative of $\ind_G(D)$. These two ingredients lead to an expression for the left hand side of \eqref{eq index}, Proposition \ref{prop ind X MA}. That proposition is a version of the commutativity of the top-right part of Diagram \ref{eq diagram proof}.

\subsection{Cocycles on $G$ and $MA$}

\begin{lemma} \label{lem cocycle}
For all $f_0, \ldots, f_l \in \cC(G)$,
\begin{multline} \label{eq cocycle}
\Phi^P_x(f_0, \ldots, f_l) = \int_{M/Z} \int_{K^{l+1}} \int_{N^{l+1}} \int_{(MA)^l} 
\det (a_1, \ldots, a_l)\\
f_0(k_0 n_0 hxh^{-1} m_1^{-1} a_1^{-1} k_1^{-1} ) 
f_1(k_1 n_1 m_1 m_2^{-1} a_1 a_2^{-1} k_2^{-1}) \cdots \\
\cdots
f_{l-1}(k_{l-1} n_{l-1} m_{l-1} m_l^{-1} a_{l-1} a_l^{-1} k_l^{-1})
f_l(k_l n_l m_l a_l k_0^{-1})
\\
dm_1 \cdots dm_l\, 
da_1 \cdots da_l\, 
dn_0 \cdots dn_{l}\,
dk_0 \cdots dk_l\, 
 d(hZ).
\end{multline}
\end{lemma}
\begin{proof}
First substituting $k^{-1}g_j k$ for $g_j$ in the right hand side of \eqref{eq def cocycle}, and then $g_j \cdots g_l$ for $g_j$, 
 we obtain
\begin{multline*}
\Phi^P_x(f_0, \ldots, f_l) = \int_{M/Z} \int_{KN} \int_{G^l} 
\det \bigl(H(g_1),  \ldots, H(g_l )\bigr)\\
f_0(khxh^{-1}ng_1^{-1}k^{-1} ) f_1(kg_1g_2^{-1}k^{-1}) \cdots
f_{l-1}(k g_{l-1} g_l^{-1}k^{-1})
 f_l(kg_lk^{-1})\\
\, dg_1 \cdots dg_n\, dk\, dn\, d(hZ).
\end{multline*}
Using \eqref{eq decomp dg} and the definition of the map $H$, we rewrite the right hand side as
\begin{multline*}
 \int_{M/Z} \int_{KN} \int_{(KMAN)^l} 
\det (a_1,  \ldots,  a_l)\\
f_0(khxh^{-1}n n_1^{-1} a_1^{-1} m_1^{-1} k_1^{-1} k^{-1} )
 f_1(k
 k_1 m_1 a_1 n_1 n_2^{-1} a_2^{-1} m_2^{-1} k_2^{-1}
 k^{-1}) \cdots \\
 \cdots
 f_{l-1}(k
 k_{l-1} m_{l-1} a_{l-1} n_{l-1} n_l^{-1} a_l^{-1} m_l^{-1} k_l^{-1}
 k^{-1})
 f_l(kk_l m_l a_l n_l k^{-1})\\
 dm_1 \cdots dm_l\, 
da_1 \cdots da_l\, 
dn\, dn_1 \cdots dn_{l}\,
dk\, dk_1 \cdots dk_l\, 
 d(hZ).
\end{multline*}
The facts that $N$ normalises $MA$  and that $M$ and $A$ commute, together with substitutions in the integrals over $N$, imply that the above expression equals
\begin{multline*}
\int_{M/Z} \int_{K^{l+1}} \int_{N^{l+1}} \int_{(MA)^l} 
\det(a_1, \ldots, a_l)\\
f_0(k n_0 hxh^{-1} m_1^{-1} a_1^{-1} k_1^{-1}k^{-1} ) 
f_1(k k_1 n_1 m_1 m_2^{-1} a_1 a_2^{-1} k_2^{-1} k^{-1}) \cdots \\
\cdots
f_{l-1}(k k_{l-1} n_{l-1} m_{l-1} m_l^{-1} a_{l-1} a_l^{-1} k_l^{-1}k^{-1})
f_l(k k_l n_l m_l a_l  k^{-1})
\\
dm_1 \cdots dm_l\, 
da_1 \cdots da_l\, 
dn_0 \cdots dn_{l}\,
dk_0 \cdots dk_l\, 
 d(hZ).
\end{multline*}
A substitution in the integrals over $K$ now gives the desired equality \eqref{eq cocycle}.
\end{proof}

If $f \in \cC(G)$, define the function $f^N$\label{pag fN} on $MA$ by
\beq{eq fN}
f^N(ma) := \int_N f(nma)\, dn.
\eeq
It was shown in Lemma 21 in \cite{HC66} that this integral converges for all $m \in M$ and $a \in A$, and that this defines a Schwartz  function $f^N \in \cC(MA)$. 

Let $\cH$ be a Hilbert space. We write $\cL^1(\cH)$ for the algebra of trace-class operators on $\cH$.
Define seminorms on the space of smooth maps from $G$ to $\cL^1(\cH)$ by taking the seminorms on $\cC(G)$ and replacing absolute values of functions by the trace norm.
Let $\cC(G, \cL^1(\cH))$ be the Fr\'echet space of all such maps for which these seminorms are finite.
%
%
For a  cyclic $l$-cocycle $\varphi$ on $\cC(G)$, let $\varphi \# \Tr\colon \bigl((\cC(G, \cL^1(\cH)) \bigr)^{l+1} \to \C$ be as in Theorem III.1.$\alpha$.12 in \cite{ConnesBook}, initially defined on the dense subspace $\cC(G) \otimes \cL^1(\cH)$ and extended continuously.
%
Then the pairing of $\varphi$ with an idempotent $q$ in a matrix algebra over $\cC(G, \cL^1(\cH))$, representing an element in $K_0\bigl(\cC(G,  \cL^1(\cH))\bigr)$, 
is given by
\beq{eq pairing Tr}
\langle \varphi, q\rangle = (\varphi \# \Tr)(q, \ldots, q).
\eeq
Here $\Tr$ denotes the tensor product of the operator trace on $\cL^1(\cH)$ and the matrix trace.
\begin{proposition} \label{prop cocycle MA}
Let $\cH$ be a Hilbert space equipped with a representation of $K$. For $j=0, \ldots, l$, let 
$f_j \in \cC(G, \cL^1(\cH))$ be such that for all $g \in G$ and $k,k' \in K$,
\beq{eq fj KK}
f_j(kgk') = k \circ f_j(g) \circ k' \quad \in \cL^1(\cH).
\eeq
Let $f_j^N$ be defined via the extension of \eqref{eq fN} to a map $\cC(G, \cL^1(\cH)) \to \cC(MA,
\cL^1(\cH))$. Then
\[
(\Phi^P_x \# \Tr)(f_0, \ldots, f_l) 
 = (\Phi^{MA}_x \# \Tr)(f_0^N, \ldots, f_l^N).
\]
\end{proposition}
\begin{proof}
Viewing $MA$ as a parabolic subgroup of itself, and applying
Lemma \ref{lem cocycle}, we find that
 that for all $f_0^{MA},\ldots, f_l^{MA} \in \cC(MA,  \cL^1(\cH))$,
\begin{multline*} 
(\Phi^{MA}_x \# \Tr)(f_0^{MA}, \ldots, f_l^{MA}) = \int_{M/Z} \int_{(K\cap M)^{l+1}}  \int_{(MA)^l} 
\det (a_1, \ldots, a_l)\\
\Tr\Bigl(
f_0^{MA}(k_0  hxh^{-1} m_1^{-1} a_1^{-1} k_1^{-1} ) 
f_1^{MA}(k_1  m_1 m_2^{-1} a_1 a_2^{-1} k_2^{-1}) \cdots \\
\cdots
f_{l-1}^{MA}(k_{l-1}  m_{l-1} m_l^{-1} a_{l-1} a_l^{-1} k_l^{-1})
f_l^{MA}(k_l  m_l a_l k_0^{-1}) \Bigr)
\\
dm_1 \cdots dm_l\, 
da_1 \cdots da_l\, 
dk_0 \cdots dk_l\, 
 d(hZ). 
\end{multline*}
Using the facts that $k_j \in M$ and that $M$ and $A$ commute, we can use substitutions to rewrite the right hand side as
\begin{multline} \label{eq cocycle 2}
 \int_{M/Z} \int_{(MA)^l} 
\det (a_1, \ldots, a_l) 
\Tr\Bigl(
f_0^{MA}( hxh^{-1} m_1^{-1} a_1^{-1} ) 
f_1^{MA}( m_1 m_2^{-1} a_1 a_2^{-1} ) \cdots 
\\
\cdots
f_{l-1}^{MA}(m_{l-1} m_l^{-1} a_{l-1} a_l^{-1} )
f_l^{MA}(m_l a_l) \Bigr)
dm_1 \cdots dm_l\, 
da_1 \cdots da_l\, 
 d(hZ). 
\end{multline}
Explicitly, if $l \geq 4$, then one substitutes $k_0 h$ for $h$;  $k_1 m_1 k_0^{-1}$ for $m_1$; $k_2 m_2 k_0^{-1}$ for $m_2$; $k_{l-1} m_{l-1} k_0^{-1}$ for $m_{l-1}$;  $k_l m_l k_0^{-1}$ for $m_l$; and $k_j m_j$ for $m_j$ if $3\leq j \leq l-2$.

Now if $f_0, \ldots, f_l \in \cC(G, \cL^1(\cH))$ satisfy \eqref{eq fj KK}, then Lemma \ref{lem cocycle} implies that 
\begin{multline*} 
(\Phi^P_x \# \Tr)(f_0, \ldots, f_l) =\\
 \int_{M/Z} \int_{K} \int_{(MA)^l} 
\det (a_1, \ldots, a_l)
\Tr\Bigl(
k_0f_0^N( hxh^{-1} m_1^{-1} a_1^{-1} ) 
f_1^N(m_1 m_2^{-1} a_1 a_2^{-1} ) \cdots \\
\cdots
f_{l-1}^N( m_{l-1} m_l^{-1} a_{l-1} a_l^{-1} )
f_l^N(m_l a_l) k_0^{-1}\Bigr)
dm_1 \cdots dm_l\, 
da_1 \cdots da_l\, 
dk_0\, 
 d(hZ).
\end{multline*}
By the trace property of the operator trace, this equals
 the right hand side of \eqref{eq cocycle 2}, with $f_j^{MA}$ replaced by $f_j^N$.
\end{proof}


\subsection{A slice}

Using Abels' slice theorem \cite{Abels}, we write
\beq{eq Abels}
X = G \times_K Y,
\eeq
for a $K$-invariant, compact submanifold $Y \subset X$.\label{pag Y} 
 The index \eqref{eq indG} is independent of the $G$-invariant Riemannian metric on $X$. From now on, we choose such a metric induced by a $K$-invariant Riemannian metric on $Y$ and a $K$-invariant inner product on $\kp$ via
\[
TX \cong G \times_K (\kp \oplus TY).
\]
We will write $dy$ (and $dy'$ and $dy_j$) for the Riemannian density on $Y$. Then we can and will normalise the Haar measure $dg$ on $G$ so that for all $f \in C_c(X)$, 
\[
\int_X f(p)\, d\vol_p = \int_G \int_Y f(gy)\, dy\, dg,
\]
where $d\vol$ is the Riemannian density on $X$. See Lemma 4.1 in \cite{HSIII}.

Suppose that the adjoint representation $\Ad\colon K \to \SO(\kp)$ lifts to a homomorphism 
\beq{eq tilde Ad}
\widetilde{\Ad}\colon K \to \Spin(\kp). 
\eeq
This is not a restriction, since this is always true for a double cover $\tilde G$ of $G$, and one can then use the fact that $\tilde G \times_{\tilde K} Y =  G \times_{K} Y  = X$, for a maximal compact subgroup $\tilde K<\tilde G$. Let $S_{\kp}$\label{pag Sp} be the $\Spin$-representation of $\Spin(\kp)$, viewed as a representation of $K$ via $\widetilde{\Ad}$.
 The slice $Y$ in \eqref{eq Abels} has a $\Spinc$-structure with spinor bundle $S_Y \to Y$\label{pag SY} such that
\beq{eq S Y}
S = G \times_K(S_{\kp} \otimes S_{Y}).
\eeq
See Proposition 3.10 in \cite{HM14}. 

Corresponding to \eqref{eq S Y}, we have the decomposition
\beq{eq decomp L2E}
L^2(E) = \bigl(L^2(G) \otimes S_{\kp} \otimes L^2(S_Y \otimes W|_Y) \bigr)^K.
\eeq
 Let $L$ be the left regular representation of $G$ on $C^{\infty}(G)$, and let $c$ be the Clifford action by $\kp$ on $S_{\kp}$.
Let $\{X_1, \ldots, X_r\}$ be an orthonormal basis of $\kp$ with respect to the Killing form.
Consider the operator\label{pag DGK}
\[
D_{G,K} := \sum_{j} L(X_j) \otimes c(X_j) 
\]
on
$
C^{\infty}(G) \otimes S_{\kp}.
$
We have
\beq{eq DGK DY}
D = D_{G,K} \otimes 1 + 1 \otimes D_Y
\eeq
on 
\[
\Gamma^{\infty}(E) \cong \bigl( C^{\infty}(G) \otimes S_{\kp} \otimes \Gamma^{\infty}(S_Y \otimes W|_Y) \bigr)^K,
\]
for a $\Spinc$-Dirac operator $D_Y$\label{pag DY} on $S_Y$, coupled to $W|_Y$. Here we use a graded tensor product, which means that $1 \otimes D_Y$ means the grading operator on $S_{\kp}$ tensored with $D_Y$. 

Consider the vector bundle $\Hom(E|_Y) := E|_Y \boxtimes E|_Y^* \to Y \times Y$. Consider the action by $K \times K$ on 
$\cC(G) \otimes \Gamma^{\infty}(\Hom(E|_Y))$ defined by
\beq{eq SE KK}
\bigl(
(k,k')\cdot (f \otimes A)\bigr)(g,y,y') = f(kgk'^{-1}) k^{-1} A(ky, k'y') k',
\eeq
for $k,k' \in K$, $f \in \cC(G)$, $A \in \Gamma^{\infty}(\Hom(E|_Y))$, $g \in G$ and $y,y' \in Y$. Let $\cC_0(E) \subset\cC(G) \otimes \Gamma^{\infty}(\Hom(E|_Y))$ be the space of elements invariant under this action. An element $\tilde \kappa \in \cC_0(E)$ defines a smooth section $\kappa$ of $\Hom(E) := E \boxtimes E^* \to X \times X$ given by
\beq{eq kappa kappa tilde}
\kappa(gy, g'y') = g \tilde \kappa(g^{-1}g', y, y')g'^{-1},
\eeq
for $g,g' \in G$ and $y,y' \in Y$. 

We define $\cC(E) := \cC\bigl(G, \cL^1(L^2(E|_Y))\bigr)$ as above Proposition \ref{prop cocycle MA}.
%
We identify elements of $\cC(E)$ with the $G$-equivariant operators they define on $L^2(E)$ via \eqref{eq kappa kappa tilde}.

\subsection{Functional calculus and operators in $\cC(E)$}

We will use the fact that the $K$-theory class $\ind_G(D)$ can be represented by idempotents that lie in the unitisation of $\cC(E)$, see Lemma \ref{lem hk SE}. That allows us to apply Proposition \ref{prop cocycle MA}, see the proof of Proposition \ref{prop ind X MA}. These arguments are based on the following fact.
\begin{proposition} \label{prop fD2}
Let  $f$ be a  Schwartz function on $\R$, and assume that for all $n \in \Z_{\geq 0}$, there are $a,b>0$ such that for all $x>0$, the $n$th derivative of $f$ at $x$ satisfies 
\beq{eq bound fn}
|f^{(n)}(x)| \leq ae^{-bx}.
\eeq
Then the operators $f(D^2)$  and $Df(D^2)$ defined by functional calculus lie in $\cC(E)$. 
\end{proposition}

For every $V \in \widehat K$, let $\spec_V(D_Y) \subset \spec(D_Y)$ be the set of eigenvalues of the restriction of $D_Y$ to the $V$-isotypical component $L^2(S_Y \otimes W|_Y)_V$ of $L^2(S_Y \otimes W|_Y)$. For $\lambda \in \spec_V(D_Y)$, let $L^2(S_Y \otimes W|_Y)_{V, \lambda} \subset L^2(S_Y \otimes W|_Y)_V$ be the corresponding eigenspace.
Then by \eqref{eq decomp L2E},
\[
L^2(E) \cong \bigoplus_{V \in \widehat K} \bigoplus_{\lambda \in \spec_V(D_Y)} \bigl( L^2(G) \otimes S_{\kp} \otimes  L^2(S_Y \otimes W|_Y)_{V, \lambda}\bigr)^K.
\]
Via the Fourier transform, the right hand side is isomorphic to
 \beq{eq decomp L2E 2}
  \int_{\hat G_{\temp}}^{\oplus}  \bigoplus_{\substack{V \in \widehat K\\ [\pi|_K:S_{\kp} \otimes V]\not=0}} \bigoplus_{\lambda \in \spec_V(D_Y)} \cH_{\pi} \otimes
 \bigl( \cH_{\pi}^*  \otimes S_{\kp} \otimes  L^2(S_Y \otimes W|_Y)_{V, \lambda}\bigr)^K\, d\mu(\pi),
 \eeq
where $\mu$ is the Plancherel measure. Every representation $\pi \in \hat G_{\temp}$ is of the form $\pi = \Ind_{P}^G(\sigma \otimes \nu \otimes 1)$, for a cuspidal parabolic subgroup $P = MAN<G$, a (limit of) discrete series representation $\sigma$ of $M$, and $\nu \in \hat A = i\ka^*$. (See  \cite{KZ1, KZ2}, or Theorem 14.91 in \cite{Knapp}.) We denote by
\[
(\Lambda(P,\sigma), \nu) \in i(\kt\cap \km)^* \otimes i\ka^*
\]
the infinitesimal character of $\Ind_{P}^G(\sigma \otimes \nu \otimes 1)$. 
Then the Casimir element $\Omega_G$ of $G$ acts on $\pi$ as the scalar
\beq{eq Omega G}
\Omega_G(\pi) = -\|\Lambda(P,\sigma)\|^2 + \|\nu\|^2.
\eeq

Let $f$ be as in Proposition \ref{prop fD2}.  Let $\widehat{f(D^2)}$ be the operator on \eqref{eq decomp L2E 2} corresponding to the operator $f(D^2)$ on $L^2(E)$.
\begin{lemma}\label{lem hat fD2}
The operator $\widehat{f(D^2)}$ is given by multiplication by a function on $\hat G_{\temp} \times \spec(D_Y)$. At $(\pi = \Ind_P^G(\sigma \otimes \nu \otimes 1), \lambda)$, with $\lambda \in \spec_V(D_Y)$, the absolute value of this function is at most
\beq{eq bound FT fD2}
a e^{-b\lambda^2} e^{-b (\|\nu\|^2 - \|\Lambda(P,\sigma)\|^2 + \|\mu_V+\rho_K\|^2)},
\eeq
where $\mu_V$ is the highest weight of $V$, and $\rho_K$ is half the sum of a choice of positive roots for $(K, T)$. All derivatives of this function with respect to $\nu$ satisfy the same estimate, for possibly different $a$ and $b$.

\end{lemma}
\begin{proof}
Because of the use of graded tensor products in \eqref{eq DGK DY}, 
\[
D^2 = D_{G,K}^2 \otimes 1 + 1 \otimes D_Y^2
\] 
where the tensor products on the right hand side are ungraded. This implies that the two terms commute, so that for all $s \in \R$,
\[
e^{isD^2} = e^{isD_{G,K}^2} \otimes e^{is D_Y^2},
\]
as in (4.3) in \cite{HW2}.
So Proposition 10.3.5 in \cite{Higson00} implies that
\[
f(D^2)= \int_{\R}e^{isD^2}\hat f(s) ds = \int_{\R}e^{isD_{G, K}^2} \otimes  e^{is D_Y^2} \hat f(s) \, ds.
\]

After Fourier transform, the operator $D_{G,K}^2$ acts on \eqref{eq decomp L2E 2} as multiplication by a function, which takes the value
\[
\|\nu\|^2 - \|\Lambda(P,\sigma)\|^2 + \|\mu_V+\rho_K\|^2
\]
at the point corresponding to $V \in \widehat K$ and $\pi = \Ind_P^G(\sigma \otimes \nu \otimes 1)$.  Here we use \eqref{eq Omega G}, and Proposition 3.1 in \cite{Parthasarathy72}; see also  (3.17) in \cite{Atiyah77}.

We find that  $\widehat{f(D^2)}$ is indeed given by multiplication by a function on $\hat G_{\temp} \times \spec(D_Y)$, and its value at  $(\Ind_P^G(\sigma \otimes \nu \otimes 1), \lambda)$, with $\lambda \in \spec_V(D_Y)$, is
\beq{eq FT fD2}
\int_{\R}e^{is(\|\nu\|^2 - \|\Lambda(P,\sigma)\|^2 + \|\mu_V+\rho_K\|^2}   e^{is \lambda^2} \hat f(s) \, ds = f(\|\nu\|^2 - \|\Lambda(P,\sigma)\|^2 + \|\mu_V+\rho_K\|^2 + \lambda^2).
\eeq
By assumption on $f$, this function and all its derivatives with respect to $\nu$ satisfy the desired estimate.
\end{proof}

 Our next goal is to estimate the decay behaviour of \eqref{eq bound FT fD2}. We use the term \emph{rapidly decreasing} to mean decreasing faster than any rational function. Let us define a function $C$ on $\widehat{K}$ by 
 \[
 C(V) = \sum_{\lambda \in \spec_V(D_Y)} e^{-b\lambda^2},
 \]
 for $V \in \widehat{K}$.

\begin{lemma}\label{lem compact case}
The function $C$ decreases rapidly in $\|\mu_V\|$. 
\end{lemma}
\begin{proof}
Consider the  special case of Proposition \ref{prop fD2} where $G = K$ is compact. Then the heat kernel of $D$ lies in 
\[
\mathcal{C}(E) = \left(C^\infty(K) \otimes \mathcal{L}^1(E|_Y) \right)^K,
\]
simply because it is smooth. Its component in $\mathcal{L}^1(E|_Y)$ is nonnegative, so its trace norm on that component equals its trace. Applying this trace on $Y$, we obtain a smooth function on $K$, whose Fourier transform equals $C$. And the Fourier transform of a smooth function on $K$ is a rapidly decreasing function on $\widehat K$.
\end{proof}

\begin{lemma}\label{lem compact case 2}
The function
\[
\widetilde{C}(x) :=\sum_{\substack{W \in \widehat K\\ \|\mu_W\| \geq x}}C(W)
\]
decreases rapidly in $x\geq 0$. 
\end{lemma}
\begin{proof}
If $x \in \N$, then
\[
\begin{split}
\widetilde{C}(x) &= \sum_{n=x}^{\infty} \sum_{\substack{W \in \widehat K\\ n \leq \|\mu_W\|< n+1}} C(W)\\
 &\leq \sum_{n=x}^{\infty} 
 \# \bigl\{W \in \widehat K; n \leq \|\mu_W\|< n+1 \bigr\}\max_{\|\mu_W\| \geq n} C(W).
\end{split}
\]
The number $ \# \{W \in \widehat K; n \leq \|\mu_W\|< n+1\}$ increases at most polynomially in $n$. And $\max_{\|\mu_W\| \geq n} C(W)$ decreases rapidly in $n$, because $C(W)$ decays rapidly in $\|\mu_W\|$ by Lemma \ref{lem compact case}. So the terms in the latter sum decay rapidly in $n$. This implies that the sum decays rapidly in $x$, for example via an estimate of the sum by an integral.
%
\end{proof}

\begin{lemma}\label{lem est phi}
%
Let $\varphi$ be the operator-valued function on $\hat G_{\temp}$ given by 
\begin{multline}\label{eq def phi}
\varphi(\pi =  \Ind_P^G(\sigma \otimes \nu \otimes 1))
 =\\
\sum_{\substack{V \in \widehat K\\ [\pi |_K:S_{\kp} \otimes V]\not=0}} 
 \sum_{\lambda \in \spec_V(D_Y)} e^{-b\lambda^2} e^{-b (\|\nu\|^2 - \|\Lambda(P,\sigma)\|^2 + \|\mu_V+\rho_K\|^2)} \cdot \mathrm{id}_{(\pi \otimes S_\kp \otimes V)^K},
 \end{multline}
for some $b>0$. Then the  function
\[
\|{\varphi}\|_{\HS} \in C^\infty(\widehat{G}_{\temp})
\]
is a rapidly decreasing function in $\|\Lambda(P, \sigma)\|$ and $\| \nu\|$, where $\| \bullet\|_{\HS}$ means the Hilbert--Schmidt norm. 
\end{lemma}
\begin{proof}
We study the decay behaviour of $\varphi$ in $\nu$ and $\sigma$ separately.

For the decay behaviour in $\nu$, we note that positivity of $D_{G,K}^2$ implies that
\beq{eq Partha}
-\|\Lambda(P, \sigma)\|^2 + \|\mu_V+ \rho_K\|^2 \geq 0.
\eeq
(This is known as Parthasarathy's Dirac inequality.) Hence for all $\sigma$, 
\beq{eq phi nu}
\varphi(\Ind_P^G(\sigma \otimes \nu \otimes 1)) \leq  \Bigl(\sum_{V \in \widehat K} C(V) [\pi|_K: S_{\kp} \otimes V]\Bigr) e^{-b \|\nu\|^2 }.
\eeq
For any $W \in \widehat {K}$, we have
\[
[\pi|_K:W] \leq \dim(W).
\]
See for example Theorem 8.1 in \cite{Knapp}. (Compare also  Theorem 2.9(3) and Example 2.12 in \cite{Kobayashi98}.)
So  $[\pi|_K: S_{\kp} \otimes V]$ grows at most polynomially in $\|\mu_V\|$, uniformly in $\pi$. And
the function $C$ decreases rapidly by Lemma \ref{lem compact case}. So the sum over $V$ in \eqref{eq phi nu} converges, and $\varphi$ decreases rapidly in $\|\nu\|$, uniformly in $\sigma$.


We now consider decay behaviour with respect to $\sigma$.
For any $\pi \in \widehat{G}$, we write
\[
{\min}_K({\pi}):= \min \{  \|\mu_V\|; V \in \widehat K, [\pi|_K:S_{\kp}\otimes V]\not=0\}.
\]
%
By Frobenius reciprocity, we have for all $W \in \widehat K$, 
\beq{eq pi K}
[\pi|_K :W]= \left[\Ind_P^G(\sigma \otimes \nu \otimes 1)\big|_K:W\right] = \left[\left(L^2(K) \otimes \sigma \right)^{M \cap K}:W \right] = \left(W^*|_{M \cap K} \otimes \sigma \right)^{M \cap K}. 
\eeq
This implies that $\|{\min}_K(\pi)\|$  increases as a positive power of  $\|\Lambda(P, \sigma)\|$ when $\|\Lambda(P, \sigma)\|$ is large enough.

Now \eqref{eq Partha} implies that
\[
\|\varphi(\Ind_P^G(\sigma \otimes \nu \otimes 1))\|_{\HS} \leq
\widetilde C({\min}_K(\pi))  [\pi|_K: S_{\kp} \otimes V].
%
\]
Because ${\min}_K(\pi)$  increases as a positive power of  $\|\Lambda(P, \sigma)\|$ and $\widetilde C$ decreases rapidly by Lemma \ref{lem compact case 2}, we find that $\widetilde C({\min}_K(\pi)) $ decreases rapidly in $\|\Lambda(P, \sigma)\|$ and is independent of $\nu$.  Hence $\|\varphi(\Ind_P^G(\sigma \otimes \nu \otimes 1))\|_{\HS}$ also decreases rapidly in  $\|\Lambda(P, \sigma)\|$, uniformly in $\nu$.
%
%
%
%
%
\end{proof}

\begin{proof}[Proof of Proposition \ref{prop fD2}.]
We first show that $f(D^2) \in \cC(E)$. We consider its Fourier transform $\widehat{f(D^2)}$ on \eqref{eq decomp L2E 2}, and apply the trace norm on the component $L^2(S_Y \otimes W|_Y)$. 
Lemmas \ref{lem hat fD2} and  \ref{lem est phi} imply that the resulting function is a Schwartz function on $\hat G_{\temp}$. By a characterisation of $\cC(G)$ by Arthur \cite{Arthur75} (based on results by Harish-Chandra \cite{HC58, HC66}),
this implies that $f(D^2)$ lies in $\cC(G) \otimes \End(S_{\kp})$ after we apply the trace norm on $L^2(S_Y \otimes W|_Y)$. Hence $f(D^2) \in \cC(E)$.

 To show that $Df(D^2) \in \cC(E)$, we write
 \[
D f(D^2)= (D_{G,K} \otimes 1)f(D^2)
%
+
(1 \otimes  D_Y)f(D^2),
 \]
 using graded tensor products. The first term on the right lies in $\cC(E)$ because the universal enveloping algebra of $\kg$ preserves $\cC(G)$, so that $D_{G,K} \otimes 1$ preserves $\cC(E)$. The second term on the right lies in $\cC(E)$ by a similar argument as for $f(D^2)$. The difference is that for $(1\otimes D_Y)f(D^2)$, the expression  \eqref{eq FT fD2} is replaced by
 \[
 \int_{\R}e^{is(\|\nu\|^2 - \|\sigma\|^2 + \|\mu_V+\rho_K^M\|^2)}\lambda  e^{is \lambda^2} \hat f(s) \, ds = \lambda f(\|\nu\|^2 - \|\sigma\|^2 + \|\mu_V+\rho_K^M\|^2 + \lambda^2).
 \]
The rest of the argument still applies, with $e^{-b\lambda^2}$ replaced by $\lambda e^{-b \lambda^2}$ everywhere.
\end{proof}

\subsection{The index on $X$}

We will use an explicit representative of the $K$-theory class $\ind_G(D)$. 
Suppose  that $G/K$ and $X$ are even-dimensional. Then $Y$ is also even-dimensional.
We will later deduce the odd-dimensional case of Theorem \ref{thm index} from the even-dimensional case; see Lemma \ref{lem even odd}.
 Let $D^{\pm}$ be the restriction of $D$ to even and odd-graded sections of $E$, respectively. Fix $t>0$, and write\label{pag qt}
\beq{eq def qt}
q_t:= \left( 
\begin{array}{cc}
e^{-tD^-D^+} & e^{-\frac{t}{2}D^- D^+} \frac{1-e^{-t D^-D^+}}{D^-D^+} D^- \\
e^{-\frac{t}{2}D^+ D^-} \frac{1-e^{-t D^+D^-}}{D^+D^-} D^+ & 1_{E^-}-e^{-tD^+ D^-} 
\end{array}
\right) \quad \in \Gamma^{\infty}(\Hom(E)),
\eeq
and\label{pag pE}
\[
p_E:= 
\left( 
\begin{array}{cc}
0&0 \\ 0 & 1_{E^-}
\end{array}
\right) \quad \in \cB(L^2(E)).
\]

\begin{lemma} \label{lem hk SE}
The operator $q_t$ lies in the unitisation of $\cC(E)$.
\end{lemma}
\begin{proof}
By Proposition \ref{prop fD2}, the operators $e^{-tD^2}$ and 
$e^{-\frac{t}{2}D^2} \frac{1-e^{-t D^2}}{D^2} D$  lie in $\cC(E)$. 
\end{proof}


The operator $q_t$ is an idempotent, so we have 
\[
[q_t] - [p_E] \in K_*(\cC(E)).
\]
For any cyclic cocycle $\varphi$ over $\cC(G)$,
\beq{eq ind qt}
\langle \varphi, \ind_G(D) \rangle = \langle \varphi \# \Tr, [q_t] - [p_E] \rangle,
\eeq
where the pairing on the right is as in \eqref{eq pairing Tr}.
See page 356 in \cite{Connes82}.
%
Combining \eqref{eq ind qt} with Proposition \ref{prop cocycle MA}, we obtain the main conclusion of this section.
\begin{proposition}\label{prop ind X MA}
We have
\[
\langle \Phi^P_x,  \ind_G(D) \rangle = \langle \Phi^{MA}_x \# \Tr, [q_t^N] - [p_E^N]\rangle.
\]
\end{proposition}
\begin{proof}
By Lemma \ref{lem hk SE} and the $K\times K$-invariance property \eqref{eq SE KK} of the elements of $\cC(E)$, the element $q_t \in \cC(G, \cL^1(L^2(E|_Y)))$ has the $K\times K$-invariance property of the functions $f_j$ in 
 Proposition \ref{prop cocycle MA}. By that proposition, \eqref{eq ind qt} and \eqref{eq pairing Tr}, the claim follows.
\end{proof}

\section{An index pairing on $X/N$}

We construct a Dirac operator on $X/N$ and give an explicit realisation of its $MA$-equivariant index, Proposition \ref{prop qtN}. This proposition is a variation on the commutativity of the top left  diagram in \eqref{eq diagram proof}.

\subsection{Operators on $X$ and $X/N$}

Consider the $MA$-invariant submanifold $X_{MA} := MA \times_{K \cap M} Y \subset X$.\label{pag XMA} Then $X = NAM \times_{K \cap M} Y$, so the inclusion map $X_{MA} \hookrightarrow X$ induces $X_{MA} \cong X/N$.
For $\kappa \in \cC(E) \cap \Gamma^{\infty}(\Hom(E))$, let 
\beq{eq kappa N}
\kappa^N \in \bigl(\cC(MA) \hat \otimes \Gamma^{\infty}(\Hom(E|_Y))\bigr)^{K\cap M \times K\cap M} \hookrightarrow
\End(E|_{X_{MA}})
\eeq
be defined as in \eqref{eq fN}. 
(We use the completed tensor product of the Fr\'echet space $\cC(MA)$ and the nuclear Fr\'echet space $ \Gamma^{\infty}(\Hom(E|_Y))$ here, instead of the space $\cC(MA, \cL^1(L^2(E|_Y)))$, to make it clear that $\kappa^N$ defines a smooth kernel on $X_{MA}$.)
Explicitly, for all $m \in M$, $a \in A$ and $y,y' \in Y$,\label{pag kappaN}
\[
\kappa^N(ma, y,y') = \int_{N} \kappa(nam, y, y')\, dn \quad \in \Hom(E_{y'}, E_y).
\]

 Consider the space  $\Gamma^{\infty}(E)^{N, c}$ of  smooth, $N$-invariant sections of $E$ whose support has compact image in $X/N$. Then $\Gamma^{\infty}(E)^{N, c} \cong \Gamma^{\infty}_c(E|_{X_{MA}})$ via restriction to $X|_{MA}$. Via this identification, a section $\kappa^N$ as in \eqref{eq kappa N} defines an operator on $\Gamma^{\infty}(E)^{N, c}$.
If $s_1$ and $s_2$ are sections of $E$ that are not necessarily square-integrable, we will denote the integral
\[
\int_X (s_1(p), s_2(p))_E\, d\vol_p
\]
by $(s_1, s_2)_{L^2(E)}$ if it converges. 
\begin{lemma} \label{lem inner prod N}
For all $\kappa \in \cC(E)$, $\sigma \in \Gamma^{\infty}_c(E)$ and $s \in \Gamma^{\infty}(E)^{N, c}$,
\beq{eq inner prod N}
(\kappa^* \sigma, s)_{L^2(E)} = (\sigma, \kappa^N s)_{L^2(E)}.
\eeq
In particular, the integral defining the left hand side converges.
\end{lemma}
\begin{proof}
For $\kappa \in \cC(E)$, let $\tilde\kappa$ be defined as in \eqref{eq kappa kappa tilde}. Writing $X = NAM \times_{K \cap M} Y$, we see that the left hand side of \eqref{eq inner prod N} equals
\begin{multline*}
\int_{NAM} \int_Y \int_{NAM} \int_Y \bigl(
\tilde \kappa^*(namy, n'a'm'y') \sigma(n'a'm'y'), s(namy)
 \bigr)_E \\
dy'\, dn'\,da'\, dm'\, 
dy\, dn\,da\, dm \\
\end{multline*}
By definition of the adjoint kernel $\tilde \kappa^*$, and by $N$-invariance of $s$, and $G$ invariance of $\tilde \kappa$, this integral equals
\begin{multline*}
\int_{NAM} \int_Y \int_{NAM} \int_Y \bigl(
 \sigma(n'a'm'y'), n \tilde \kappa(n^{-1}n'a'm'y', amy)s(amy)
 \bigr)_E \\
dy'\, dn'\,da'\, dm'\, 
dy\, dn\,da\, dm.
\end{multline*}
Substituting $n'^{-1}n$ for $n$,  we rewrite this as
\begin{multline} \label{eq inner prod 2}
\int_{NAM} \int_Y \int_{NAM} \int_Y 
\bigl(
 \sigma(n'a'm'y'), n' n
\tilde \kappa(n^{-1}a'm'y', amy)
s(amy)
 \bigr)_E 
 \\
dy'\, dn'\,da'\, dm'\, 
dy\, dn\,da\, dm. 
\end{multline}
If $\tilde \kappa^N$ is related to $\kappa^N$ as in \eqref{eq kappa kappa tilde}, with $G$ replaced by $MA$, then for all $y,y' \in Y$, $a,a' \in A$ and $m,m' \in M$, unimodularity of $N$ implies that
\[
\int_N n
\tilde \kappa(n^{-1}a'm'y', amy)\, dn = \tilde \kappa^N(a'm'y', amy).
\]
So \eqref{eq inner prod 2} equals
\[
 \int_{NAM} \int_Y 
\bigl(
 \sigma(n'a'm'y'), n' 
(\tilde \kappa^N s)(a'm'y')
 \bigr)_E 
dy'\, dn'\,da'\, dm'\, 
\]
By $N$-invariance of $\tilde \kappa^N s$, we have $n' 
(\tilde \kappa^Ns)(a'm'y') = 
(\tilde \kappa^Ns)(n'a'm'y')$ for all $y' \in Y$, $n' \in N$, $a' \in A$ and $m' \in M$. So the latter integral equals the right hand side of \eqref{eq inner prod N}.
\end{proof}

Let $\kappa_t \in \cC(E^+)$ be kernel of the heat operator $e^{-tD^-D^+}$. 
\begin{lemma}\label{lem inner products}
For all $\sigma \in \Gamma^{\infty}_c(E^+)$ and 
$s \in \Gamma^{\infty}(E^+)^{N, c}$, and
for all $t>0$,
\[
\frac{d}{dt}(\sigma, \kappa_t^N s)_{L^2(E^+)} = (\sigma, -D^-D^+ e^{-t D^- D^+}  s)_{L^2(E^+)}.
\]
\end{lemma}
\begin{proof}
By Lemma \ref{lem inner prod N}, 
\[
(\sigma, \kappa_t^N s)_{L^2(E^+)} = (e^{-t D^- D^+} \sigma, s)_{L^2(E^+)}.
\]
The derivative with respect to $t$ at $t$ equals
\beq{eq inner prod 1}
(-D^-D^+ e^{-t D^- D^+} \sigma, s)_{L^2(E^+)}.
\eeq
The operator $-D^-D^+ e^{-t D^- D^+}$ is symmetric on sections in $L^2(E^+)$, so we have to take the fact into account that $s$ is not in $L^2(E^+)$.

Let $\varepsilon > 0$. For $r>0$, let $\psi_r \in C^{\infty}_c(M)$ be constant $1$ within a distance $r$ of $\supp(\sigma)$. Then \eqref{eq inner prod 1} equals
\[
\bigl(-D^-D^+ e^{-t D^- D^+} \sigma, \psi_r s\bigr)_{L^2(E^+)}  -\bigl( (1-\psi_r)D^-D^+ e^{-t D^- D^+} \sigma, s\bigr)_{L^2(E^+)}.
\]
Because $\sigma$ and $\psi_r s$ lie in $\Gamma^{\infty}_c(E^+)$, the first term equals 
$( \sigma, -D^-D^+ e^{-t D^- D^+} \psi_r s)_{L^2(E^+)} $.  And the absolute value of the second term is at most equal to 
\[
\|(1-\psi_r)D^-D^+ e^{-t D^- D^+} \sigma\|_{L^1(E^+)} \|s\|_{\infty}.
\]
The first factor is smaller than $\varepsilon$ for $r$ large enough by Gaussian decay properties of heat kernels. (Here we also use the fact that $X/G$ is compact, so $X$ has bounded geometry and volumes of balls in $X$ are bounded by an exponential function of their radii.) For such a value of $r$, we find that
\[
\bigl| (-D^-D^+ e^{-t D^- D^+} \sigma, s)_{L^2(E^+)} - ( \sigma, -D^-D^+ e^{-t D^- D^+} \psi_r s)_{L^2(E^+)}\bigr| < \varepsilon \|s\|_{\infty}.
\]
We similarly have
\begin{multline*}
\bigl|  ( \sigma, -D^-D^+ e^{-t D^- D^+} s)_{L^2(E^+)} - 
  ( \sigma, -D^-D^+ e^{-t D^- D^+} \psi_r s)_{L^2(E^+)} \bigr|\\
  = 
 |  \bigl( (1-\psi_r)  D^-D^+ e^{-t D^- D^+} \sigma, s\bigr)_{L^2(E^+)} |< \varepsilon \|s\|_{\infty}.
\end{multline*}
We conclude that \eqref{eq inner prod 1} equals $ ( \sigma, -D^-D^+ e^{-t D^- D^+} s)_{L^2(E^+)}$, and the claim follows.
\end{proof}


\subsection{An index on $X/N$}\label{sec index XN}

We have a $K \cap M$-invariant decomposition
\[
\kp = (\kp \cap \km) \oplus \ka \oplus ((\kp \cap \km) \oplus \ka)^{\perp},
\]
where the orthogonal complement is taken inside $\kp$ with respect to 
the Killing form. 
The homomorphim $\widetilde{\Ad}$ maps $K \cap M$ into $\Spin(\kp \cap \km) \times \Spin(\ka) \times \Spin(((\kp \cap \km) \oplus \ka)^{\perp}) \hookrightarrow \Spin(\kp)$. 
Let $S_{\kp \cap \km}$,\label{pag Spm} $S_{\ka}$\label{pag Sa} and 
 $S_{ ((\kp \cap \km) \oplus \ka)^{\perp}}$ be the corresponding $\Spin$-representations of $K \cap M$. (The group $K \cap M$ acts trivially on $S_{\ka}$, because $M$ centralises $A$.) Then we have a decomposition of  representations of $K \cap M$,
\beq{eq decomp Sp}
S_{\kp} = S_{\kp \cap \km} \otimes S_{\ka} \otimes S_{ ((\kp \cap \km) \oplus \ka)^{\perp}}.
\eeq
If $G/K$ is even-dimensional, then this includes the gradings on the respective spaces, and we use graded tensor products. 

\begin{lemma} \label{lem S perp}
There is an isomorphism of graded representations of $K \cap M$
\[
S_{ ((\kp \cap \km) \oplus \ka)^{\perp}} \cong S_{\kk/(\kk \cap \km)}.\label{pag Skm}
\]
\end{lemma}
\begin{proof}
There are $T \cap M$-equivariant linear isomorphisms
\[
\kp \cong \kg/\kk \cong \km/(\kk \cap \km) \oplus \ka \oplus \kn \cong (\kp \cap \km) \oplus \ka \oplus \kn. 
\]
If $\theta$ is the Cartan involution, then the map $y\mapsto (y+\theta y)/2 + \kk \cap \km$ is a $T \cap M$-invariant isomorphism from $\kn$ onto $\kk/(\kk \cap \km)$. So $((\kp \cap \km) \oplus \ka)^{\perp} \cong \kk/(\kk \cap \km)$ as representations of $T \cap M$. This implies that their characters are equal on $K \cap M$, so they are equal as representations of $K \cap M$ as well.
\end{proof}

Consider the $MA$-equivariant vector bundles\label{pag SN}\label{pag FN}
\beq{eq SN FN}
\begin{split}
S_N &:= MA \times_{K \cap M} (S_{\kp \cap \km} \otimes S_{\ka} \otimes S_Y) \to X_{MA} = X/N; \\
F_N &:= MA \times_{K \cap M} (Y \times S_{\kk/(\kk \cap \km)}) \to X_{MA} = X/N.
\end{split}
\eeq
\begin{lemma}\label{lem SN Spinc}
The vector bundle $S_N$ is the spinor bundle of an $MA$-equivariant $\Spinc$-structure on $X/N$, and \eqref{eq SN} holds.
\end{lemma}
\begin{proof}
Clifford multiplication is an algebra bundle isomorphism $\Cl(TX) \to \End(S)$, where $\Cl$ stands for the  Clifford bundle. The Clifford bundle
\[
\Cl(TX_{MA}) = \Cl(MA \times_{K \cap M} ( \kp \cap \km \oplus \ka \oplus TY))
\]
maps precisely onto $\End(S_N)$ under this isomorphism. So $S_N$ is the spinor bundle of an $MA$-equivariant $\Spinc$-structure on $X_{MA}$. Lemma \ref{lem S perp} and \eqref{eq S Y} imply \eqref{eq SN}.
\end{proof}
The equality \eqref{eq SN} implies that 
\beq{eq EXMA}
E|_{X_{MA}} = MA \times_{K \cap M} \bigl((S_{\kp \cap \km}  \otimes S_{\ka}) \otimes (S_Y \otimes S_{\kk/(\kk \cap \km)}  \otimes W|_Y)\bigr).
\eeq

Let $\{X_1, \ldots, X_s\}$ be a basis of $(\kp \cap \km) \oplus \ka$, orthonormal with respect to the Killing form. Let $L$ be the left regular representation of $MA$.
Let\label{pag DMAK} 
\beq{eq Dirac MA}
D_{MA/(K \cap M)} := \sum_{j} L(X_j) \otimes c(X_j) 
\eeq
be the $\Spin$-Dirac operator on the bundle 
\beq{eq SMA}
MA \times_{K \cap M} (S_{\kp \cap \km} \otimes S_{\ka}) \to MA/(K \cap M).
\eeq
It extends to an operator $D_{MA, K \cap M}$\label{pag DMAKM}  on 
\[
 C^{\infty}(MA) \otimes (S_{\kp \cap \km}  \otimes S_{\ka}) ,
\]
defined by the same expression \eqref{eq Dirac MA}.

 Let $D_{Y,M}$\label{pag DYM}  be the $\Spinc$-Dirac operator on $Y$ coupled to $S_{\kk/(\kk \cap \km)} \otimes W|_Y$, acting on sections of the bundle
\[
S_Y \otimes S_{\ka} \otimes S_{\kk/(\kk \cap \km)} \otimes  W|_Y \to Y.
\]
Consider the operator\label{pag DXMA}
\beq{eq DXMA}
D_{X_{MA}} := D_{MA, K \cap M} \otimes 1 + 1 \otimes D_{Y,M},
\eeq
where we use graded tensor products; i.e.\ $1 \otimes D_{Y,M}$ means the grading operator on \eqref{eq SMA} tensored with $D_{Y,M}$.
Initially, it acts in the space
\[
C^{\infty}(MA) \otimes (S_{\kp \cap \km}  \otimes S_{\ka}) \otimes \Gamma^{\infty}(S_Y \otimes S_{\kk/(\kk \cap \km)} \otimes W|_Y),
\]
but we view it as acting on the subspace of $K \cap M$-invariant elements, which is the space of smooth sections of  \eqref{eq EXMA}. Then $D_{X_{MA}}$ is a $\Spinc$-Dirac operator on $X_{MA}$ for the $\Spinc$-structure with spinor bundle $S_N$, 
  twisted by the vector bundle $W_N := F_N \otimes W/N$.

As before, let $\kappa_t \in \cC(E^+)$ be Schwartz kernel of the heat operator $e^{-tD^-D^+}$, and let $\kappa_t^N$ be defined as in \eqref{eq kappa N}. 
\begin{lemma} \label{lem kappa t N}
The Schwartz kernel of $e^{-t D^-_{X_{MA}} D^+_{X_{MA}}  }$ is  $\kappa_t^N$. The analogous statement for $e^{-tD^+_{X_{MA}} D^-_{X_{MA}}}$ holds as well.
\end{lemma}
\begin{proof}
%
%
%
%
The operator $D$ commutes with the action of $\kappa^N_t$ on $\Gamma^{\infty}(E^+)^{N, c}$ because it is $N$-equivariant. So Lemmas \ref{lem inner products} and \ref{lem inner prod N} imply that for all $\sigma \in \Gamma^{\infty}_c(E^+)$ and 
$s \in \Gamma^{\infty}(E^+)^{N, c} \cong \Gamma^{\infty}_c(E^+|_{X_{MA}})$, and all $t>0$,
\[
\begin{split}
\frac{d}{dt}(\sigma, \kappa_t^N s)_{L^2(E^+)} &= (e^{-t D^- D^+}  \sigma, -D^-D^+  s)_{L^2(E^+)}\\
& = ( \sigma, - \kappa_t^N D^-D^+  s)_{L^2(E^+)} \\
& = ( \sigma, -  D^-D^+  \kappa_t^N s)_{L^2(E^+)}. 
\end{split}
\]\
On $N$-invariant sections, $D$ equals $D_{X_{MA}}$, so the latter inner product equals
$( \sigma, -  D_{X_{MA}}^-D_{X_{MA}}^+  \kappa_t^N s)_{L^2(E^+)}$.


Lemma \ref{lem inner prod N}  also implies that 
\[
\lim_{t\downarrow 0}(\sigma, \kappa_t^N s)_{L^2(E^+)}  = (\sigma, s)_{L^2(E^+)}. 
\]
We conclude that the map $(p, t)\mapsto (\kappa_t^N s)(p)$ satisfies the heat equation for the operator $D_{X_{MA}}^-D_{X_{MA}}^+  $, with the boundary condition $\lim_{t \downarrow 0}( \kappa_t^N s) (p)= s(p)$ for all $p \in X_{MA}$. 
%

The statement for $e^{-tD_{X_{MA}}^+D_{X_{MA}}^-}$ can be proved analogously.
\end{proof}

\begin{lemma} \label{lem lambda t N}
Let $\lambda_t \in \cC(E)$ be the smooth kernel of $e^{-\frac{t}{2}D^- D^+} \frac{1-e^{-t D^-D^+}}{D^-D^+} D^-$. Then $\lambda_t^N$ is the Schwartz kernel of $e^{-\frac{t}{2}D_{X_{MA}}^- D_{X_{MA}}^+} \frac{1-e^{-t D_{X_{MA}}^-D_{X_{MA}}^+}}{D_{X_{MA}}^-D_{X_{MA}}^+} D^-_{X_{MA}}$. The analogous statement for $e^{-\frac{t}{2}D^+ D^-} \frac{1-e^{-t D^+D^-}}{D^+D^-} D^+$ holds as well.
\end{lemma}
\begin{proof}
Fix $t_0 \in (0,t)$. As in Lemma \ref{lem hk SE}, the Schwartz kernel $\nu_t$ of the operator 
\[
 e^{-\frac{t_0}{2}D^- D^+} \frac{1-e^{-t D^-D^+}}{D^-D^+} D^-
\]
lies in $\cC(E)$.
%
 %
 Similarly to the proof of Lemma \ref{lem kappa t N}, we use Lemmas \ref{lem inner prod N} and \ref{lem inner products}  to find that for all $\sigma \in \Gamma^{\infty}_c(E^+)$ and 
$s \in \Gamma^{\infty}(E^-)^{N, c}$,
\[
\begin{split}
\frac{d}{dt} (\sigma, \nu_t^N s)_{L^2(E^+)} &= 
\frac{d}{dt} \Bigl(
 D^+ \frac{1-e^{-t D^-D^+}}{D^-D^+} e^{-\frac{t_0}{2}D^-D^+}
\sigma,  s \Bigr)_{L^2(E^-)} \\
&= ( e^{-\frac{t_0}{2}D^-D^+} e^{-tD^- D^+} D^+ 
\sigma,  s)_{L^2(E^-)} \\
&= (
\sigma,  e^{-\frac{t_0}{2}D^-D^+} e^{-tD^- D^+} D^- s)_{L^2(E^+)} \\
&= (\sigma, e^{-\frac{t_0}{2}D_{X_{MA}}^-D_{X_{MA}}^+} e^{-tD_{X_{MA}}^- D_{X_{MA}}^+} D_{X_{MA}}^- s)_{L^2(E^+)}.
\end{split}
\]
So $\nu_t^N$ has the defining properties of the operator 
\[
e^{-\frac{t_0}{2}D_{X_{MA}}^-D_{X_{MA}}^+} \frac{1-e^{-t D_{X_{MA}}^- D_{X_{MA}}^+} } 
{D_{X_{MA}}^-D_{X_{MA}}^+} D_{X_{MA}}^-.
\] 

Lemma \ref{lem kappa t N} implies that if $\kappa_t$ is the Schwartz kernel of $e^{-\frac{t-t_0}{2}D^-D^+}$, then $\kappa_t^N$ is the Schwartz kernel of $e^{-\frac{t-t_0}{2}D_{X_{MA}}^-D_{X_{MA}}^+}$.
Because $N$ normalises $M$ and $A$, the map $\kappa \mapsto \kappa^N$ is multiplicative. 
So 
$
\lambda_t^N = \kappa_t^N \nu_t^N, 
$
which by the preceding arguments is the Schwartz kernel of $e^{-\frac{t}{2}D_{X_{MA}}^- D_{X_{MA}}^+} \frac{1-e^{-t D_{X_{MA}}^-D_{X_{MA}}^+}}{D_{X_{MA}}^-D_{X_{MA}}^+} D^-_{X_{MA}}$.

The  statement for $e^{-\frac{t}{2}D^+ D^-} \frac{1-e^{-t D^+D^-}}{D^+D^-} D^+$ can be proved analogously.
\end{proof}

\begin{proposition} \label{prop qtN}
For any cyclic cocycle $\varphi$ over $\cC(MA)$,
\[
\langle \varphi, \ind_{MA}(D_{X_{MA}}) \rangle = \langle \varphi \# \Tr, [q_t^N] - [p_{E|_{X_{MA}}}] 
\rangle.
\]
\end{proposition}
\begin{proof}
Lemmas \ref{lem kappa t N} and \ref{lem lambda t N} imply that $q_t^N$ equals \eqref{eq def qt}, with $D$ replaced by $D_{X_{MA}}$. So the claim follows from the analogue of \eqref{eq ind qt} for $D_{X_{MA}}$. 
\end{proof}

Propositions \ref{prop ind X MA} and \ref{prop qtN}, and the fact that $p_E^N = p_{E|_{X_{MA}}}$, lead to the following  key step in the proof of Theorem \ref{thm index}.
\begin{proposition} \label{prop ind G MA}
We have
\[
\langle \Phi^P_x, \ind_G(D)\rangle = \langle \Phi^{MA}_x, \ind_{MA} (D_{X_{MA}}) \rangle. 
\]
\end{proposition}


\section{Proof of the index theorem}

Having established an analogue of commutativity of the top two diagrams in \eqref{eq diagram proof} with Proposition \ref{prop ind G MA}, we finish the proof of Theorem \ref{thm index} by considering the bottom two diagrams  in \eqref{eq diagram proof}.
We note that the bottom left part commutes in Lemma \ref{lem decomp DMA}, and show that the bottom right part commutes in Lemmas \ref{lem pairing product psi} and \ref{lem phi A DA}. Then applying the main result of \cite{HW2} to the action by $M$ on $X/AN$ finishes the proof.

\subsection{Cococyles for $M$ and $A$} \label{sec cocycles MA}

Consider the cyclic $l$-cocycle $\tilde \Phi^M_x$ on $\cC(M)$ given by\label{pag tilPhiMx}
\begin{multline*}
\tilde \Phi^M_x(f_0^M, \ldots, f_l^M) =  \int_{M/Z} \int_{(K\cap M)^{l+1}}  \int_{M^l} \\
f_0^{M}(k_0  hxh^{-1} m_1^{-1} a_1^{-1} k_1^{-1} ) 
f_1^{M}(k_1  m_1 m_2^{-1} k_2^{-1}) 
\cdots\\
\cdots
f_{l-1}^{M}(k_{l-1}  m_{l-1} m_l^{-1}k_l^{-1})
f_l^{M}(k_l  m_l k_0^{-1})
dm_1 \cdots dm_l\, 
dk_0 \cdots dk_l\, 
 d(hZ). 
\end{multline*}
Note that $\tilde \Phi^M_x$ is not the same as the cocycle $\Phi^M_x$ for the cuspidal parabolic $M<M$ as defined in \eqref{eq def cocycle}, because the latter is a degree zero cocycle. We do have the following fact.
\begin{lemma} \label{lem tilde phiMx}
Let $\cH$ be a Hilbert space with a unitary representation of $K$.
For all $q \in K_*((\cC(M, \cL^1(\cH)))^{K \times K})$,
\[
\langle \tilde \Phi^M_x, q\rangle = \langle  \Phi^M_x, q\rangle. 
\]
\end{lemma}
\begin{proof}
Let
$q \in M_r((\cC(M, \cL^1(\cH)))^{K \times K})$ be an idempotent. Similarly to the last paragraph of the proof of Proposition \ref{prop cocycle MA}, $K\times K$-invariance of $q$ and the  trace property of $\Tr$ imply that
\begin{multline*}
\langle \tilde \Phi^M_x, q\rangle = (\tilde \Phi^M_x \# \Tr)(q, \ldots, q) 
=
\int_{M/Z}   \int_{M^l} 
\Tr\bigl(
q(hxh^{-1} m_1^{-1} a_1^{-1} ) 
q( m_1 m_2^{-1}) 
\cdots \\
\cdots
q( m_{l-1} m_l^{-1})
q( m_l ) \bigr)
dm_1 \cdots dm_l\, 
 d(hZ) 
 = \int_{M/Z}  
\Tr(
q^{l+1}(hxh^{-1} ))
 d(hZ).
\end{multline*}
Because $q$ is an idempotent, this equals
\[
\int_{M/Z}  
\Tr(q(hxh^{-1} )) d(hZ) = \langle\Phi^M_x, q\rangle.
\]

\end{proof}



Let $X_{MA}$, $S_Y$, $S_{\kp \cap \km}$ and $S_{\kk/(\kk \cap \km)}$ be as
above Lemma \ref{lem kappa t N}.
Let $X_M := M \times_{K\cap M} Y \subset X_{MA}$.\label{pag XM} Then $X_{MA} = X_M \times A$, and $X_{M} \cong X/AN$. Consider the $M$-equivariant vector bundles
\beq{eq SAN FAN}
\begin{split}
S_{AN} &= M \times_{K \cap M}(S_{\kp \cap \km} \otimes S_Y) \to X_M \cong X/AN; \\
F_{AN} &= M \times_{K \cap M} (Y \times S_{\kk/(\kk \cap \km)}) \to X_M \cong X/AN.
\end{split}
\eeq
Analogously to Lemma \ref{lem SN Spinc},  $S_{AN}$ is the spinor bundle of an $M$-equivariant $\Spinc$-structure on $X/AN$, and Lemma \ref{lem S perp} and \eqref{eq S Y} imply that \eqref{eq SAN} holds.

Let $D_A$\label{pag DA} be the $\Spin$-Dirac operator on $A$, acting on sections of the spinor bundle $A \times S_{\ka}$. Let $D_{M/(K \cap M)}$\label{pag DMK} be the $\Spin$-Dirac operator on $M/(K\cap M)$, acting on sections of the spinor bundle $M \times_{K \cap M} S_{\kp \cap \km}$. Let $D_{M, K \cap M}$\label{pag DMKM} be its extension to an operator on $C^{\infty}(M) \otimes S_{\kp\cap \km}$, defined analogously to the operator $D_{MA, K \cap M}$ in Subsection \ref{sec index XN}. Analogously to \eqref{eq DXMA}, consider the operator\label{pag DXM}
\[
D_{X_M} := D_{M, K\cap M} \otimes 1 + 1 \otimes D_{Y,M}
\]
on the space  of $K \cap M$-invariant elements of
\[
C^{\infty}(M) \otimes S_{\kp \cap \km} \otimes \Gamma^{\infty}(S_Y \otimes S_{\kk/(\kk \cap \km)} \otimes W|_Y),
\]
which is the space of smooth sections of $S_{AN} \otimes F_{AN} \otimes W/AN$.
Then $D_{X_M}$ is a $\Spinc$-Dirac operator on $X_M$ for the $\Spinc$-structure with spinor bundle $S_{AN}$, 
twisted by the vector bundle  $W_{AN} = F_{AN} \otimes W/AN$.
\begin{lemma} \label{lem decomp DMA}
We have
\[
\ind_{MA}(D_{X_{MA}}) = \ind_M(D_{X_M}) \otimes \ind_A(D_A).
\] 
On the right hand side, we used the exterior Kasparov product
\[
 KK(\C, C^*_rM) \times KK(\C, C^*_rA) \to KK(\C, C^*_r(MA)).
 \]
\end{lemma}
\begin{proof}
We have
\[
D_{X_{MA}} = 
D_{X_M} \otimes 1 + 1 \otimes D_A,
\]
where we use graded tensor products. Hence the claim follows by
Theorem 5.2 in \cite{Hochs09}.
\end{proof}

	
\begin{lemma}\label{lem pairing product psi}
\[
\langle \Phi^{MA}_x, \ind_{MA} (D_{X_{MA}}) \rangle= \langle \Phi^M_x, \ind_{M}(D_{X_M}) \rangle \langle \Phi^A_e, \ind_{A}(D_A)\rangle.
\]
\end{lemma}
\begin{proof}
Applying
Lemma \ref{lem cocycle} and with $G$ and $P$ both replaced by $MA$, so that the integrals over $K$ and $N$ are not present, we find that $\Phi^{MA}_x$ decomposes as
\[
\Phi^{MA}_x = \Phi^{M}_x \cup \Phi^{A}_e.
\]
So by Lemma
 \ref{lem decomp DMA}, the pairing $\langle \Phi^{MA}_x, \ind_{MA} (D_{X_{MA}}) \rangle$ equals
\beq{eq pairing MA}
\langle \tilde{\Phi}^M_x,  \ind_{M}(D_{X_M}) \rangle \langle \Phi^A_e, \ind_{A}(D_A)\rangle.
\eeq
This follows from compatibility of the cup product on cyclic cohomology with the external Kasparov product, or from a direct computation involving the special case of Lemma \ref{lem cocycle} just mentioned. The claim then follows
by Lemma \ref{lem tilde phiMx}.
\end{proof}

\begin{lemma} \label{lem phi A DA}
We have
\[
\langle \Phi^A_e, \ind_A D_A\rangle = 1.
\]
\end{lemma}
\begin{proof} This is a very special case of Theorem 4.6 in \cite{Pflaum15}. 
\end{proof}

\begin{lemma} \label{lem even odd}
Let $D_{X \times \R}$ be the $\Spinc$-Dirac  operator on $X \times \R$ for the $\Spinc$-structure with spinor bundle $S \times \R$,
  twisted by $W \times \R$. Then
\[
\langle \Phi^{P \times \R}_{(x,0)}, \ind_{G \times \R} (D_{X \times \R})\rangle = \langle \Phi^P_x, \ind_G(D)\rangle.
\]
\end{lemma}
\begin{proof}
The operator $D_{X \times \R}$ equals
\[
D_{X \times \R} = D \otimes 1 + 1 \otimes i\frac{d}{dt}
\]
on $\Gamma^{\infty}(E \times \R) = \Gamma^\infty(E) \otimes C^{\infty}(\R)$.
So by Lemmas \ref{lem pairing product psi} and \ref{lem phi A DA}, and 
an analogue of Lemma \ref{lem decomp DMA},
\[
\langle \Phi^{P \times \R}_{(x,0)}, \ind_{G \times \R} (D_{X \times \R})\rangle 
= 
\langle \Phi^P_x, \ind_G(D)\rangle
\langle \Phi^{\R}_0, \ind_{\R}(i\frac{d}{dt})\rangle
= \langle \Phi^P_x, \ind_G(D)\rangle.
\]
\end{proof}

\subsection{Proofs of  Theorem \ref{thm index} and Lemma \ref{lem index XN}} \label{sec pf index}

\begin{lemma} \label{lem char Skm}
The character of $S_{\kk/(\kk \cap \km)}$ restricted to $T \cap M$ equals
\beq{eq char Skm}
\prod_{\lambda} (e^{\lambda/2} - e^{-\lambda/2}),
\eeq
where $\lambda$ runs over the $T \cap M$-weights of $\kk/(\kk \cap \km)$.  If $P$ is not a maximal cuspidal parabolic, then this equals zero.
\end{lemma}
\begin{proof}
The first claim follows from Remark 2.2 in \cite{Parthasarathy72}.
If $P$ is not a maximal cuspidal parabolic, then $T \cap M$ is not a maximal torus in $K$. Hence there is a nonzero subspace  of $\kk$ on which $T \cap M$ acts trivially. This means that the weight zero occurs in \eqref{eq char Skm}, which is therefore zero.
\end{proof}

\begin{proof}[Proof of Theorem \ref{thm index}]
If $X$ and $G/K$ are even-dimensional, then 
applying Proposition \ref{prop ind G MA} 
and Lemma 
 \ref{lem pairing product psi},
we see that
\beq{eq ind G ind M}
\langle \Phi^P_x, \ind_G(D)\rangle = \langle \Phi^{MA}_x, \ind_{MA} (D_{X_{MA}})\rangle 
= \langle \Phi^{A}_e, \ind_{A}( D_{A})\rangle  \langle \Phi^{M}_x, \ind_{M} (D_{X_{M}})\rangle.
\eeq
Lemma \ref{lem phi A DA} implies that the right hand side equals 
\beq{eq ind M}
\langle \Phi^{M}_x, \ind_{M} (D_{X_{M}})\rangle.
\eeq
Because $M$ has a compact Cartan subgroup, Proposition 4.11 in \cite{HW2} and the fact that $X/AN$ is a smooth manifold equal to $X_M$ imply that \eqref{eq ind M} equals the right hand side of \eqref{eq index} if $x$ lies in a compact subgroup of $M$, and is zero otherwise. Here we use the fact that $\Phi^M_x$ equals the orbital integral trace $\tau_x$ used in \cite{HW2}. Two points to note here are the following.
\begin{itemize}
\item
The condition in Proposition 4.11 in \cite{HW2} that $x$ has `finite Gaussian orbital integral'  is in fact always satisfied, see Section 4.2 in \cite{BismutHypo}.
\item In \cite{HW2}, it was assumed for simplicity that the group acting is connected. The group $M$ is not connected in general, but it belongs to Harish-Chandra's class, which is enough for the results in \cite{HW2} to hold. 
\end{itemize}

If $P$ is not a maximal cuspidal parabolic, then Lemma \ref{lem char Skm} implies that 
 the bundle $F_{AN}$ 
in the second line of \eqref{eq SAN FAN}
is zero as a virtual vector bundle. So the virtual vector bundle $W_{AN}$ is zero, and $\ch([W_{AN}|_{\supp(\chi_x)}](x))$ on the right hand side of \eqref{eq index} is zero as well.

By Lemma \ref{lem even odd},  the case where $X$ and $G/K$ are even-dimensional implies the case where $X$ and $G/K$ are odd-dimensional. If the dimensions of $X$ and $G/K$ have different parities (i.e.\ $Y$ is odd-dimensional), then both sides of \eqref{eq index} are zero.
%
%
\end{proof}

\begin{proof}[Proof of Lemma \ref{lem index XN}]
Recall the definition of the projection maps $q_A\colon X/N \to A$ and $q\colon X/N \to X/AN$ above Lemma \ref{lem index XN}.
The decomposition $X/N = X/AN \times A$ implies that
\[
\begin{split}
\hat A(X/N) &= q^*\hat A(X/AN)q_A^*\hat A(A) = q^*\hat A(X/AN);\\
W_N &= q^*W_{AN};\\
S_N &= q^*S_{AN} \otimes S_{\ka}.
\end{split}
\]
The third equality implies that
 $L_{\det}^{N} = q^*L_{\det}$. 

Let $P$ be a maximal cuspidal parabolic, and let $x \in T < K \cap M$. Then $x$ acts trivially on $A$, so $(X/N)^x = (X/AN)^x \times A$. This implies that $\cN_{N} = q^*\cN$. We conclude that 
\eqref{eq index XN} equals
\begin{multline*}
\int_{(X/N)^x} q^* \left(\chi_x \frac{\hat A((X/AN)^x) \ch([W_{AN}|_{\supp(\chi_x)}](x)) e^{c_1(L_{\det}|_{(X/AN)^x})}}{\det(1-x e^{-R^{\cN}/2\pi i})^{1/2}}\right) \wedge q_A^*( \chi_A\, da) \\
=\int_{(X/AN)^x}  \chi_x \frac{\hat A((X/AN)^x) \ch([W_{AN}|_{\supp(\chi_x)}](x)) e^{c_1(L_{\det}|_{(X/AN)^x})}}{\det(1-x e^{-R^{\cN}/2\pi i})^{1/2}}
\int_A\chi_A\, da,
\end{multline*}
which equals the right hand side of \eqref{eq index}.
\end{proof}

\section{Extensions and special cases}

\subsection{General elliptic operators}

By standard arguments, the index formula \eqref{eq index} for twisted $\Spinc$-Dirac operators implies an index formula for general elliptic operators. We sketch the argument here, and refer to \cite{BvEII} for some details. Also compare this with the proof  of Theorem 2.5 in \cite{HW2}.

Let $E \to X$ be any Hermitian $G$-vector bundle, and $D$ a $G$-equivariant elliptic differential operator on $E$.
Let $\Sigma X$ be the manifold obtained from $X$ by glueing two copies of the unit ball bundle in $TX$ along the unit sphere bundle. This has a natural $G$-invariant almost complex structure. Let $p^{\Sigma X}\colon \Sigma X \to X$ be the projection map. There is a Hermitian $G$-vector bundle $V_D \to \Sigma X$ defined in terms of the principal symbol of $D$, such that the $\Spinc$-Dirac operator $D_{\Sigma X}^{V_D}$ on $\Sigma X$ (for the $\Spinc$-structure defined by the almost complex structure) twisted by $V_D$ satisfies
\[
[D] = p^{\Sigma X}_* [D_{\Sigma X}^{V_D}] \in K_0^G(X),
\]
see Theorem 5.0.4 in \cite{BvEII}. Naturality of the analytic assembly map implies that
\[
\ind_G(D) = \ind_G(D_{\Sigma X}^{V_D}),
\]
and Theorem \ref{thm index} is a topological expression for the pairing of the index on the right hand side with $\Phi^P_x$.

\subsection{A higher $L^2$-index theorem}

The case of Theorem \ref{thm index}  where $x=e$ is a higher version of Wang's $L^2$-index theorem, 
Theorem 6.10 in \cite{Wang14}, in the case of linear reductive Lie groups. This is the equality 
\beq{eq higher L2}
\langle \Phi^P_e, \ind_G(D)\rangle 
= 
\int_{X/AN} \chi_e^{AN} \hat A(X/AN) \ch([W_{AN}|_{\supp(\chi_e)}] ) e^{c_1(L_{\det})}.
\eeq
This implies
 a higher version of Connes and Moscovici's $L^2$-index theorem on homogeneous spaces, Theorem 5.2 in \cite{Connes82}.
 
Let $H < K \cap M$ be a closed subgroup. On page 309 of \cite{Connes82}, a Chern character
\[
\ch\colon R(H) \to H^*(\km, H)
\]
is defined, where $H^*(\km, H)$ denotes relative Lie algebra cohomology. Furthermore, an $\hat A$-class $\hat A(\km, H) \in H^*(\km, H)$ is constructed there from the representation $\km/\kh$ of $H$.
Let $V$ be a finite-dimensional virtual representation of $H$. 
 Suppose, for simplicity, that $G/H$ has a $G$-invariant $\Spin$-structure. (This assumption may be dropped as described on page 307 of \cite{Connes82}.) Let $D_{G/H}^V$ be the $\Spin$-Dirac operator on $G/H$ coupled to $G \times_H V \to G/H$.
\begin{corollary} \label{cor CM}
We have
\beq{eq CM}
\langle \Phi^P_e, \ind_G(D_{G/H}^V)\rangle = \int_{\km/\kh} \hat A(\km, H) \ch(V \otimes S_{\kk/(\kk \cap \km)}).
\eeq
\end{corollary}
\begin{proof}
It is shown in Corollary 6.14 and Remark 6.15 in \cite{Wang14} that if $X=G/H$ and $D = D_{G/H}^V$, the right hand side of \eqref{eq higher L2} equals the right hand side of \eqref{eq CM}.
\end{proof}

If $G$ has a compact Cartan subgroup, then $M=G$, and Corollary \ref{cor CM} reduces to Theorem 5.2 in \cite{Connes82}. If $G$ does not have a compact Cartan subgroup, then both sides of the equality in Theorem 5.2 in \cite{Connes82} equal zero by Theorem 6.1 in \cite{Connes82} and Harish-Chandra's criterion $\rank(G)=\rank(K)$ for existence of discrete series represenrations; see also (1.2.5) in \cite{BM}. Corollary \ref{cor CM} is a generalisation of  Theorem 5.2 in \cite{Connes82} that gives a nontrivial result even when $G$ does not have a compact Cartan subgroup; see also Subsection \ref{sec complex}.

\subsection{Dirac induction}

We assume that $G/K$ is equivariantly $\Spin$ for simplicity. As before, let $R(K)$ be the representation ring of $K$. The Dirac induction map
\[
\DInd_K^G\colon R(K) \to K_*(C^*_r(G))
\]
 from the Connes--Kasparov conjecture maps an irreducible representation  $V$  of $K$ to $\ind_G(D_{G/K}^V)$, where $D_{G/K}^V$ is the $\Spin$-Dirac operator on $G/K$ coupled to $G \times_K V \to G/K$. In this case, Theorem \ref{thm index} implies that
\beq{eq pairing GK}
\langle \Phi^P_x, \ind_G(D_{G/K}^V) \rangle = (-1)^{\dim(M/(K \cap M))/2} \frac{\chi_V (x) \chi_{S_{\kk/(\kk \cap \km)}}(x)}{\chi_{S_{\kp \cap \km}} (x)},
\eeq
where the letter $\chi$ denotes the character of a representation. This can be deduced directly from Theorem \ref{thm index} as in the proof of Theorem 3.1 in \cite{HW19}, but the easiest was to deduce this equality is to use
the equalities \eqref{eq ind G ind M} and then apply Theorem 3.1 in \cite{HW19} directly.

The following fact was deduced from a fixed-point formula as Corollary 4.5 in \cite{HW19}, in the case where $G$ has a compact Cartan subgroup. We now generalise this argument to arbitrary linear reductive groups.
\begin{corollary}
Dirac induction for $G$ is injective.
\end{corollary}
\begin{proof}
Let $y \in R(K)$, and suppose that $\DInd_K^G(y)=0$.
As in the proof of Corollary 4.5 in \cite{HW19}, it follows from \eqref{eq pairing GK} that the character of $y$ is zero on $T$. Because $T$ is a maximal torus in $G$, it follows that $y = 0$.
\end{proof}


Another consequence of \eqref{eq pairing GK} is that pairing with $\Phi^P_x$ detects all information about  classes in $K_*(C^*_r(G))$ for maximal $P$, and no information for non-maximal $P$.
\begin{corollary}\label{cor sep pts}
If $P$ is a maximal cuspidal  parabolic subgroup, then pairing with $\Phi^P_x$ for regular $x \in K \cap M$ separates points in $K_*(C^*_r(G))$.
If $P$ is  not a maximal cuspidal  parabolic subgroup, then the pairing of  $\Phi^P_x$ with any class in $K_*(C^*_r(G))$   is zero. 
\end{corollary}
\begin{proof}
The first claim follows from \eqref{eq pairing GK} as in Corollary 4.1 in \cite{HW19}. Here one uses surjectivity of Dirac induction.
 For the second claim, we use the fact that $S_{\kk/(\kk  \cap \km)}$ is zero as a virtual representation of $K \cap M$, if $P$ is not maximal (see Lemma \ref{lem char Skm}). By surjectivity of Dirac induction, \eqref{eq pairing GK} implies that the pairing of $\Phi^P_x$ with any class in $K_*(C^*_r(G)) = K^*(\cC(G))$ is zero. 
\end{proof}
The first part of Corollary \ref{cor sep pts} means that Theorem \ref{thm index} is a complete topological description of $\ind_G(D)$, with no loss of information.

\begin{remark}
We expect that for non-maximal $P$, the class of $\Phi^P_x$ in the cyclic cohomology of $\cC(G)$ is zero. This is a stronger statement than the second part  of Corollary \ref{cor sep pts}. Proving the vanishing of these cohomology classses will likely involve detailed information on the structure of $\cC(G)$ in terms of representation theory, as obtained in \cite{Arthur75, CCH16}. Vanishing of the pairings of these classes with elements of $K_*(C^*_r(G))$, as in  Corollary \ref{cor sep pts}, follows directly from Theorem \ref{thm index}, and does not require knowledge of the precise structure of $\cC(G)$.
\end{remark}

In Definition 5.3 in \cite{ST19}, a generator $Q_V$ of $K_*(C^*_r(G))$ is defined  independently of Dirac induction. 
Theorem 5.4 in \cite{ST19} implies that 
\beq{eq ST}
\langle \Phi^P_x, Q_V \rangle  = (-1)^{\dim(A)} \frac{\chi_V(x) \Delta^K_T(x)}{\chi_{S_{\kp \cap \km}} (x) \Delta^{K \cap M}_{T} (x)},
\eeq
where $\Delta^K_T$ and $\Delta^{K \cap M}_{T}$ are Weyl denominators for choices of positive roots for $(K, T)$ and $(K \cap M, T)$, respectively. (Note that $T \subset M$ for a maximal parabolic.)
The character of $S_{\kk/(\kk \cap \km)}$ restricted to $T \cap M$ equals \eqref{eq char Skm}, which also equals $\frac{\Delta^K_T}{\Delta^{K \cap M}_{T}}$. 
Because of this equality and \eqref{eq pairing GK} and \eqref{eq ST}, we find that for regular $x \in K \cap M$,
\[
\langle \Phi^P_x, \ind_G(D_{G/K}^V) \rangle = (-1)^{\dim(M/(K \cap M))/2+\dim(A)}\langle \Phi^P_x, Q_V \rangle.
\]
If $P$ is a maximal cuspidal parabolic, then the first part of Corollary \ref{cor sep pts} allows us to deduce that 
%
%
\[
\ind_G(D_{G/K}^V) =  (-1)^{\dim(M/(K \cap M))/2+\dim(A)} Q_V,
\] 
i.e.\ the generators defined by Dirac induction equal the generators $Q_V$, up to a sign.

\subsection{Non-cuspidal parabolics} \label{sec cuspidal}

If $P<G$ is a parabolic subgroup but not cuspidal, then the pairing on the left hand side of \eqref{eq index} is still defined. And as in \eqref{eq ind G ind M} and \eqref{eq ind M}, this pairing equals $\langle \Phi_x^M, \ind_{M}(D_{X_M})\rangle$. Because $\Phi^M_x$ is now the orbital integral trace on $\cC(M)$, Theorem 3.2(b) in \cite{HW19} implies that this pairing is zero. Hence for non-cuspidal $P$, the left hand side of \eqref{eq index} is zero. As in the proof of Corollary \ref{cor sep pts}, this fact and surjectivity of Dirac induction imply that pairing with $\Phi^P_x$ is the zero map on $K$-theory.

Now suppose that $G$ has discrete series representations. Consider the parabolic subgroup $P = G<G$. Then, because  $\Phi^G_x$ is  the orbital integral trace,   Proposition 5.1 in \cite{HW2} implies that for all regular $x \in T$,  the pairing of $\Phi^G_x$ with the $K$-theory class of a discrete series representation $\pi$ is the value of the character of $\pi$ at $x$, and in particular nonzero for some $x$. By the previous paragraph, this implies that $G$ is a cuspidal parabolic subgroup of itself.  (Here we do not need the full surjectivity of Dirac induction, just the fact that discrete series classes can be realised by Dirac induction.) And that means that $G$ has a compact Cartan subgroup. Thus we recover one half of Harish-Chandra's result that $G$ has a discrete series if and only if it has a compact Cartan. (This  could also have been deduced directly from Theorem 3.2(b) in \cite{HW19}.)


The authors have not checked in detail if this result by Harish-Chandra  was used in the proofs of any of the results from representation theory that we have applied, or in the proofs of any of the results we refer to. So we do not  claim here  that our proof of necessity of existence of a compact Cartan for the existence of the discrete series is independent of Harish-Chandra's. But at  least this is a possibly interesting illustration of the links of our results to representation theory.

\subsection{Complex semisimple groups}\label{sec complex}

Now suppose that $G$ is a complex semisimple group. Then the maximal parabolic subgroup $P$ is also minimal, i.e.\ a Borel subgroup. It is a classical fact that now $K/T \cong G/P$, as complex manifolds. And $M=T$, so we have an $M$-equivariant diffeomorphism
\[
X/AN = X_M = M \times_{K \cap M} Y = Y.
\]
So by 
 Proposition \ref{prop ind G MA} 
and Lemmas 
 \ref{lem pairing product psi}
and \ref{lem phi A DA}, 
\beq{eq ind complex}
\langle \Phi^{P}_x, \ind_{G} (D)\rangle = 
\ind_{T} (D_{Y})(x).
\eeq
The right hand side is the classical $T$-equivariant index of $D_Y$ evaluated at $x$. This is given by the Atiyah--Segal--Singer fixed point formula \cite{Atiyah68,ASIII}.

As an example, suppose that $D = D_{G/H}^V$ as in \eqref{eq CM}. Then \eqref{eq ind complex} implies that \eqref{eq CM} becomes
%
\[
\langle \Phi^P_e, \ind_G(D_{G/H}^V)\rangle = 
\ind (D_{K/H}^V),
\]
for a Dirac operator $D_{K/H}^V$ on $K/H$ coupled to $K \times_H V$. The right hand side is nonzero in many cases, for example if $H=T$, and $V$ is such that $\ind_K(D_{K/T}^V)$ is an irreducible representation of $K$ by the Borel--Weil--Bott construction. Then $\ind (D_{K/H}^V)$ is the dimension of this representation, up to a sign. This shows that Corollary \ref{cor CM} is nontrivial, also in the non-equal rank case.
\begin{example}
Suppose that $G = \SL(2, \C)$, $K = \SU(2)$ and $H = T = \U(1)$. Let $D_{G/T}$ be the Dolbeault--Dirac operator on $G/T$ coupled to $G \times_T \C_n$, where
$\C_n$ is the representation of $\U(1)$ in $\C$ with weight $n \in \Z_{\geq 0}$. Then for regular $x = e^{i\alpha} \in T$, the Borel--Weil theorem and \eqref{eq ind complex} imply that 
\[
\langle \Phi^{P}_x, \ind_{G} (D_{G/T})\rangle = \frac{\sin( (n+1)\alpha)}{\sin(\alpha)},
\]
the character of the irreducible representation  of $\SU(2)$ of dimension $n+1$, evaluated at $x$.
\end{example}

\section{Notation}

The numbers in brackets indicate on what page each symbol was introduced.
\begin{itemize}
\item Groups: $G$ (\pageref{pag G}), $K$ (\pageref{pag K}), $P$ (\pageref{pag P}), $M$ (\pageref{pag M}), $A$ (\pageref{pag A}), $N$ (\pageref{pag N}), $Z$ (\pageref{pag Z}), $T$ (\pageref{pag T});
\item Manifolds: $X$ (\pageref{pag X}), $Y$ (\pageref{pag Y}), $X_{MA}$ (\pageref{pag XMA}), $X_M$ (\pageref{pag XM});
\item Clifford modules and $\Spin$-representations: $S$ (\pageref{pag S}),  $S_{AN}$ (\pageref{pag SAN}), $S_Y$ (\pageref{pag SY}), $S_{\kp}$ (\pageref{pag Sp}), $S_{\ka}$ (\pageref{pag Sa}), $S_{\kp \cap \km}$ (\pageref{pag Spm}), $S_{\kk/(\kk \cap \km)}$ (\pageref{pag Skm}), $S_N$ (\pageref{pag SN});
\item Vector bundles: $W$ (\pageref{pag W}), $E$ (\pageref{pag E}),  $W_{AN}$ (\pageref{pag WAN}),  $F_{AN}$ (\pageref{pag FAN}), $\cN$ (\pageref{pag cN}), $L_{\det}$ (\pageref{pag Ldet}), $L_{\det}^N$ (\pageref{pag LdetN}), $W_N$ (\pageref{pag WN}), $F_N$ (\pageref{pag FN});
\item Dirac operators: $D$ (\pageref{pag D}), $D_{G, K}$ (\pageref{pag DGK}), $D_Y$ (\pageref{pag DY}),  $D_{X_{MA}}$ (\pageref{pag DXMA}), $D_{MA/(K\cap M)}$ (\pageref{pag DMAK}), $D_{MA, K\cap M}$ (\pageref{pag DMAKM}), $D_{Y,M}$ (\pageref{pag DYM}), $D_{X_M}$ (\pageref{pag DXM}), $D_{M/(K \cap M)}$ (\pageref{pag DMK}), $D_{M, K \cap M}$   (\pageref{pag DMKM});
\item Cyclic cocycles: $\Phi^P_x$ (\pageref{pag PhiPx}),  $\tilde \Phi^M_x$ (\pageref{pag tilPhiMx});
\item Functions and maps: $H$ (\pageref{pag H}), $\cC(G)$ (\pageref{pag CG}),  $\chi_x$ (\pageref{pag chix}), $\chi_A$ (\pageref{pag chiA}), $q$ (\pageref{pag q}), $q_A$ (\pageref{pag qA}), 
$f^N$ (\pageref{pag fN}), $\kappa^N$ (\pageref{pag kappaN});
\item Miscellaneous:  $l$ (\pageref{pag l}), $x$ (\pageref{pag x}), $q_t$ (\pageref{pag qt}), $p_E$ (\pageref{pag pE}).
\end{itemize}

\bibliographystyle{plain}

\bibliography{mybib}

\end{document}